\documentclass[11pt]{amsart}
\usepackage{amscd,amsmath,amssymb,amsfonts}
\usepackage{color}
\usepackage[cmtip, all]{xy}
\usepackage{graphicx}
\usepackage{url}
\usepackage{enumerate}
\usepackage{hyperref}
\usepackage{lineno}

\newtheorem{thm}[equation]{Theorem}
\newtheorem{struc}[equation]{Structure}
\newtheorem{lemma}[equation]{Lemma}

\newtheorem{prop}[equation]{Proposition}

\newtheorem{definition}[equation]{Definition}
\newtheorem{conjecture}[equation]{Conjecture}
\newtheorem{cor}[equation]{Corollary}
\newtheorem{hyp}[equation]{Hypothesis}
\newtheorem{question}[equation]{Question}
\newtheorem{defn}[equation]{Definition}
\newtheorem{conj}[equation]{Conjecture}

\newtheorem{property}[equation]{Property}
\newtheorem*{thm*}{Theorem}
\theoremstyle{remark}

\theoremstyle{remark}
\newtheorem{remark}[equation]{Remark}

\numberwithin{equation}{section}

\newcommand{\isoarrow}{{~\overset\sim\longrightarrow~}}

\newcommand{\CO}{{\mathcal{O}}}

\newcommand{\bJ}{{\mathbf{J}}}

\newcommand{\Erig}{{\mathcal{E}^{\rm rig}}}

\newcommand{\hG}{{\hat{G}}}

\newcommand{\ZZ}{{\mathbb Z}}

\newcommand{\Gm}{{\mathbb G}_m}
\newcommand{\bG}{{\mathbf G}}

\newcommand{\CG}{{\mathcal{G}}}

\newcommand{\CA}{{\mathcal{A}}}

\newcommand{\CL}{{\mathcal{L}}}
\newcommand{\CLs}{{\mathcal{L}^{ss}}}

\newcommand{\G}{{\Gamma}}

\newcommand{\ra}{{~\rightarrow~}}

\newcommand{\CH}{{\mathcal{H}}}

\newcommand{\Qlb}{{\overline{\mathbb Q}_{\ell}}}
\newcommand{\FF}{{\mathbb F}}
\newcommand{\Fl}{{\FF_\ell}}
\newcommand{\Zl}{{\ZZ_\ell}}
\newcommand{\F}{{\Phi}}
\newcommand{\br}{{\bar{\rho}}}
\newcommand{\Fp}{{{\mathbb F}_p}}

\newcommand{\ad}{{\mathbf A}}

\newcommand{\CC}{{\mathbb C}}

\newcommand{\NN}{{\mathbb N}}

\begin{document}

\author{Michael Harris}
\thanks{The author was partially supported by NSF Grants DMS-2001369 and DMS-1952667 and by a Simons Foundation Fellowship, Award number 663678.}

\address{Michael Harris\\
Department of Mathematics, Columbia University, New York, NY  10027, USA}
 \email{harris@math.columbia.edu}

\date{\today}

\title[Local Langlands correspondences]{Local Langlands correspondences}

\maketitle

\begin{abstract}  Let $G$ be a connected reductive group over the non-archimedean local field $F$, with residue field $k$ of characteristic
$p$.  Denote by $\F(G/F)$ the set of  equivalence classes of
Weil-Deligne parameters for $G/F$ and let $\Pi(G/F)$ denote the set of (equivalence classes of) irreducible admissible representations of $G(F)$, both
taken over a fixed algebraically closed field $C$ of characteristic $0$.  In its simplest form, the local Langlands conjecture asserts the existence of a canonical parametrization of one set by the other set.
\vskip .5 cm

\noindent{\bf Conjecture:}  
{\it \begin{itemize}
\item[(a)]  There is a canonical parametrization
$$\CL = \CL_{G/F}:  \Pi(G/F) \ra \F(G/F).$$
\item[(b)]  For any $\varphi \in  \F(G/F)$, the {\bf $L$-packet} $\Pi_\varphi := \CL^{\varphi}$ is finite.
\item[(c)]  For any $\varphi \in  \F(G/F)$ the $L$-packet $\Pi_\varphi$ is non-empty.
\end{itemize}}
\vskip .5 cm

The first part of this article is a review of the properties expected of any   correspondence that aims to be considered ``canonical," and of known results that establish some
or all of these properties for specific groups.   In the absence of compatibility with a global correspondence it is not known in general that this list of desirable properties
suffices to characterize the correspondence.  The remainder of the article outlines elements of a strategy to address points (b) and (c) in particular, for the specific parametrization
$\CL$ constructed by Genestier and Lafforgue when $F$ is of positive characteristic.

\end{abstract}

\tableofcontents


\section{Introduction}

Let $G$ be a connected reductive group over the non-archimedean local field $F$ and let $\pi$ be an irreducible admissible representation of $G(F)$. The local Langlands conjecture posits that to such a $\pi$ can be attached a parameter $\CL(\pi)$, which is an equivalence class of homomorphisms from the Weil-Deligne group $WD_F$ with values in the Langlands $L$-group ${}^LG$ over an appropriate algebraically closed field $C$ of characteristic $0$.  The conjecture also predicts that every parameter whose restriction to the Weil group $W_F$ is semisimple, in the sense recalled below, is in the image of the correspondences, and proposes a secondary parametrization of the representations with given parameter.  I review the precise statement of the current form of the conjecture below.

Recent years have seen tremendous progress toward a complete solution of this conjecture.  On the one hand, there are now canonical candidates for
a {\it semisimple Langlands parametrization}.  A homomorphism $\varphi: W_F \ra {}^LG(C)$ is called {\it semisimple} if, whenever the image of $\varphi$, intersected with the Langlands dual group $\hat{G}(C)$, is contained in a parabolic subgroup $P \subset \hat{G}(C)$, then it is contained in a Levi subgroup of $P$.
When $F$ is of positive characteristic $p$ and $C = \Qlb$, with $\ell \neq p$, then Genestier and Lafforgue have defined a semisimple parameter, $\CL^{ss}(\pi): W_F \ra {}^LG(C)$ which is well-defined up to $\hat{G}(C)$-conjugacy \cite{GLa}; their construction is purely local but is compatible with Lafforgue's global
parametrization using moduli stacks of $G$-shtukas.    When $F$ is a $p$-adic field, on the other hand, Fargues and Scholze have similarly
defined  $\CL^{ss}(\pi): W_F \ra {}^LG(\Qlb)$, derived from a categorical action of a certain category of sheaves on a stack of local Langlands parameters
on a different category of sheaves on a moduli stack of $G$-bundles on the Fargues-Fontaine curve \cite{FS}. 

At roughly the same time, Kaletha \cite{Ka19}  proposed an explicit parametrization of the {\it regular} and {\it nonsingular} supercuspidal representations $\pi$.  It was proved by Fintzen \cite{Fi21}, improving earlier work of J.-L Kim \cite{Kim}, that these representations -- definitions will be given below -- are all obtained by a construction due to J.-K. Yu \cite{Yu01} provided $p$ is prime to the order $|W|$ of the Weyl group of $G$.  (Yu's construction is valid when $G$ splits over a tamely ramified extension of $F$.) The parameters Kaletha attaches to regular and non-singular supercuspidals are even {\it irreducible} -- the image is contained in no proper parabolic subgroup, and indeed account for all the irreducible parameters, provided $p$ is prime to $|W|$.  If you imagine a classification of the representation theory of reductive groups $G$ over local fields on the model of the periodic table, with the columns labelled by primes heading off to infinity to the right, and the rows corresponding to  Dynkin diagrams, then Kaletha's proposal completely determines the theory of irreducible parameters, leaving only a few boxes unfilled on the left-hand side of the diagram.  

On the other hand, there has been a consensus for some time that the  local Langlands correspondence has been established for most classical groups.
The starting point is the proof of the correspondence for supercuspidal representations of $GL(n,F)$, first for local fields of positive characteristic \cite{LRS93,Laf02}, then for $p$-adic fields in several stages \cite{H97,H98,BHK,HT01,He00}.  All of these proofs are global in nature and realize the supercuspidal correspondence in the cohomology of a local version of a modular variety -- Drinfel'd modular varieties in the case of positive characteristic, and Shimura varieties in the $p$-adic case.  
In particular, the proofs imply that the correspondence has a purely local characterization.  
(The reduction of the correspondence to the supercuspidal case had previously been established by Zelevinsky \cite{Z80}.) 
This construction was revisited more recently using a different method \cite{Sch13}.  By that time Arthur had completed his construction of a local Langlands correspondence for symplectic and quasi-split special orthogonal  groups over $p$-adic fields \cite{A}, a result soon extended to unitary groups over
$p$-adic fields \cite{Mok,KMSW} and to split classical groups over local fields of (most) positive characteristic \cite{GV}.  Arthur's proof is 
based on the twisted stable trace formula, and is therefore also global in nature, but again the correspondence can be characterized purely locally.  Trace formula methods had much earlier been used to parametrize representations of inner forms of $GL(n)$ in terms of representations of $GL(n)$ \cite{R,DKV}.

The natural question, which remains unsolved in general, is whether all these parametrizations give the same answer.  
A more fruitful way to look at this question is whether the local Langlands correspondence can be characterized by properties intrinsic to the 
group $G(F)$, independently of the method of its construction.  Langlands himself proposed such a characterization only when $\pi$ is a {\it spherical}
representation, in which case the parameter $\varphi$ is unramified (this is Property \ref{p7a} below).  This is the logical starting point for the entire Langlands 
program, but it doesn't explain what to do about general $\pi$.   The only genuinely intrinsic general characterization is the one Henniart \cite{He93} established
for $GL(n)$, based on $L$ and $\varepsilon$-invariants of pairs of representations of $GL(n)$ and $GL(m)$, for $m < n$.  These invariants are defined
intrinsically and independently, and Henniart proved that there is only one collection of correspondences, as $m$ varies, for which these invariants match and that satisfy several other natural properties.  The constructions based on the twisted trace formula, starting with \cite{A}, make use of additional intrinsic invariants of representations $\pi$ -- the {\it distribution characters}, which genuinely determine representations  uniquely, but they rely on Henniart's characterization for $GL(n)$; as far as I know there is no characterization of the correspondence for classical groups that is based exclusively on intrinsic properties of the groups themselves.

The characterizations of the more recent constructions are based on completely different principles.  Because they are cohomological in nature, and specifically are realized by means of the cohomology of geometric objects that are defined purely locally, the constructions in  \cite{GLa,FS} define canonical correspondences.
Moreover, the constructions in \cite{GLa} for local fields of positive characteristic are compatible with the global Langlands parametrization defined by
Lafforgue \cite{Laf18}, and this characterizes them uniquely; no such characterization is known for the parametrizations of \cite{FS}, except in cases that can be related to Shimura varieties.  This immediately raises the following questions, which is the primary motivation for this paper:

\begin{question}\label{question1}  \begin{itemize}
\item[(a)] Are the local parametrizations of \cite{GLa} and \cite{FS} consistent with the (canonical) local Langlands correspondences for $GL(n)$ and its inner forms, listed above?

\item[(b)] Are the local parametrizations of \cite{GLa} and \cite{FS} consistent with the local correspondences for classical groups defined in \cite{A,Mok,KMSW,GV}?

\item[(c)]  Do the local parametrizations of \cite{GLa} and \cite{FS} coincide with the parametrization of regular and nonsingular supercuspidals proposed by Kaletha?
\end{itemize}
\end{question}

I use the word ``consistent" in (a) and (b) because the parameters of \cite{GLa,FS} are semisimple, and don't recover the full Weil-Deligne parameter.  The answer in (a) is affirmative.  In positive characteristic this follows from the compatibility with \cite{Laf18} and the compatibility of the latter with the known global Langlands correspondence for function fields, proved in \cite{Laf02} and reproved in \cite{Laf18}.  For $p$-adic fields this is Theorem IX.7.4 of \cite{FS}, which directly relates
the parametrization of that paper to the one defined in \cite{HT01,Sch13}.  It should be possible to verify (b) in a similar way, although no one seems to have 
worked out the details.  

The situation with (c) is quite different.   To indicate clearly
what is at stake,  let $\CL^{ss}$ denote any map from the set of irreducible admissible representations $\pi$ of the group $G(F)$ to the set of semisimple
Langlands parameters for $G$.  I draw attention to two properties that have been expected of the local Langlands correspondence since it was first formulated:

\begin{question}\label{surjectivity}  Suppose $G$ is quasi-split.
Let $\varphi:  W_F \ra {}^LG(C)$ be a semisimple Langlands parameter.  Is it in the image of $\CL^{ss}$?
\end{question}

When $G$ is not quasi-split one expects precisely those $\varphi$ that are {\it relevant}, in the sense of \cite{Borel},  it to be in the image of $\CL^{ss}$.

\begin{question}\label{packets}  Let $\varphi:  W_F \ra {}^LG(C)$ be a semisimple Langlands parameter and let $\Pi(\varphi)$ denote the set
of irreducible admissible representations $\pi$ of $G(F)$ such that $\CL^{ss}(\pi) = \varphi$.  Is $\Pi(\varphi)$ a finite set?
\end{question}

A much more precise version of Question \ref{packets} is formulated in Conjecture \ref{Vogan} below.

By standard facts about the decomposition of parabolically induced representations, and by the compatibility of \cite{GLa,FS} with parabolic induction,
it suffices to check these questions when $\varphi$ is irreducible, which should be in the image of supercuspidal $\pi$ -- although, significantly, not all
supercuspidal $\pi$ have or are expected to have irreducible semisimple parameters.
For groups for which J.-K. Yu's construction of supercuspidal representations is exhaustive -- when $G$ is tamely ramified and when $p \nmid |W|$, in other words, thanks to \cite{Fi21} -- Kaletha's parametrization gives an affirmative answer to both questions, including for exceptional groups, and also, significantly, for groups that are not quasi-split.

The cohomological constructions of \cite{GLa} and \cite{FS}, on the other hand, answers neither of these questions.  In fact, nothing in these constructions
guarantees that $\CL^{ss}(\pi)$ is even ramified when $\pi$ is a (regular) supercuspidal. \footnote{The unipotent supercuspidals, which exist for groups other than 
$GL(n)$ and its inner forms, can have unramified semisimple parameters but are expected to have ramified Weil-Deligne parameters; we address this point in 
\S \ref{pureWD}.} 
Even though only the exceptional groups are concerned, this points to a striking gap in our knowledge.  A second motivation for this paper is to suggest that the use of the definite article in the title of \cite{FS} -- {\it Geometrization of {\bf the} local Langlands correspondence} -- is somewhat premature.

A third motivation, naturally related to the second one, is to collect the known results on the local Langlands correspondence (for non-archimedean local fields)
in a single place.  I have attempted to provide a reasonably complete bibliography, but in at least one respect I will be falling far short of completeness.  To explain what is missing I need to make a digression.  In addition to the local Langlands parametrization, which is meant to characterize representations
of $G(F)$ by Galois-theoretic data, I have already alluded to two intrinsic invariants of admissible irreducible representations $\pi$ of $G(F)$.   
The first are the {\it distribution characters}, linear form $\chi_\pi$ on the space $C_c(G(F))$ of locally constant compactly supported functions on $G(F)$ that is defined for all $\pi$ and that satisfy the kinds of linear independence properties familiar from the representation theory of finite groups.  
The second are the $L$- and $\varepsilon$-factors, which are defined in terms of the invariant theory of the group and which are not available for all $\pi$ or even for all $G$ (none is known for $E_8$, for example).  The theory of {\it types} is a third procedure to define invariants of $\pi$.  
A type is a representation of a well-chosen open compact subgroup $K$ of $G(F)$ whose presence in the restriction of $\pi$ to $K$ can serve to situate $\pi$ within a well-defined class, and in the best situation can determine $\pi$ uniquely.  
The construction of J.-K. Yu on which Kaletha's parametrization is based starts with a type and obtains a unique $\pi$.  
One of the main results in the representation theory of $p$-adic groups is the construction by Bushnell and Kutzko of a complete theory of types that uniquely characterize supercuspidal representations of $GL(n)$.   This naturally leads to

\begin{question}\label{typesandLanglands}  Can the local Langlands parameter of a supercuspidal representation $\pi$ of $G(F)$ be determined from its type, and
vice versa?
\end{question}

For regular and non-singular supercuspidals a complete solution has been proposed by Kaletha \cite{Ka19,Ka20}.  This includes the case of $GL(n,F)$ when $n$ is not divisible by the residue characteristic $p$ of $F$, where Kaletha's parametrization is known to recover the local Langlands correspondence; more generally, Oi and Tokimoto
recently showed that Kaletha's parametrization coincides with the known local parametrization for $GL(n)$ for all $n$, at least for regular supercuspidals \cite{OT}.  In their monograph \cite{BH06} Bushnell and Henniart have given an independent and purely local construction of the local Langlands correspondence for $GL(2)$ based on types.  
The literature on types is extensive.  In this survey I will refer to a number of specific results in the area, but I will not pretend to be exhaustive, and I apologize to the many authors whose work I have not cited.  

A fourth motivation, which is perhaps my main excuse for writing this paper, is to report on some recent results I have obtained with collaborators 
in attempting to understand Questions \ref{surjectivity} and \ref{packets}.  Although these two questions seem to point in opposite directions, there is reason to
believe that their solutions are interdependent.  This should not come as a surprise.  The original proofs of the local Langlands conjectures
for $GL(n)$ \cite{LRS93,HT01,He00} all relied on Henniart's numerical correspondence \cite{He88}.   Using the Jacquet-Langlands correspondence of supercuspidal
representations of $GL(n,F)$ with a subset of the irreducible representations of the multiplicative group of a division algebra, Henniart was able to partition
the former set into finite subsets, where the partition is indexed by a positive integer called the {\it conductor} which is part of the information contained in
the $\varepsilon$-factors of pairs.  The set of irreducible Galois parameters can also be partitioned by  conductor, and Henniart was able to show, by an argument
involving the $\ell$-adic Fourier transform in positive characteristic, that the two sets have the same cardinality.  This numerical argument provides a direct
link between surjectivity (Question \ref{surjectivity}) and injectivity (Question \ref{packets}) in the case of $GL(n)$, and in particular implies the important

\begin{thm}\label{unramgln} A representation $\pi$ of $GL(n,F)$ such that $\CL^{ss}(\pi)$ is unramified is necessarily a subquotient of an unramified principal series representation.  
\end{thm}

Scholze's main innovation in \cite{Sch13} was a geometric proof of Theorem \ref{unramgln}.  Combined with the observation, implicit in \cite{He88}, and explained in more detail in \cite{BHK}, that this is the key step in the proof of the local correspondence, this provided the first proof that was independent of Henniart's numerical correspondence.   

Theorem \ref{unramgln} is the key step in the proof for $GL(n)$ because it provides a way to reconstruct all supercuspidal representations by means of the theory 
of local base change.   This is explained in more detail in \S \ref{pureWD}, where it is also explained why the analogue of Theorem \ref{unramgln} is simply false for groups 
other than $GL(n)$ and its inner forms.   Nevertheless, as of this writing, I am convinced that local base change will provide a substitute for Theorem \ref{unramgln} 
that will be central in solving Question \ref{packets}.   The reasons for this, and possible approaches to Question \ref{surjectivity}, will form the subject of Part  \ref{Strategy}.

My final motivation is naturally to express my deep appreciation of Steve Kudla for our many years of friendship and collaboration.  It was in the course of my
work with Steve on the theta correspondence that I first began thinking seriously about the Langlands correspondence for $p$-adic groups.  My paper with Steve
and Jay Sweet contained the first original contributions to the question for which I was partially responsible.  Although this paper has nothing to say about the theta
correspondence, my approach to representation theory continues to reflect Steve's influence, which goes well beyond what can be seen in our published work.

\subsection*{Acknowledgments}  I thank Rapha\"el Beuzart-Plessis, Matt Emerton, Jessica Fintzen, Guy Henniart, Jean-Pierre Labesse, Bertrand Lemaire, Luis Lomel\'i, Colette Moeglin,  and Marie-France Vign\'eras for helpful discussions and comments on various earlier versions of this material.  I especially thank Jean-Loup Waldspurger for  his comments on an earlier version of \S \ref{incorrigi} and Tasho Kaletha for his careful reading of the entire manuscript.  Finally, I am grateful to the referee(s) for correcting certain misconceptions and for suggestions that helped to clarify the exposition.

\subsection*{A word about categorical correspondences}  There is a sense in which the local Langlands correspondence presented here is an exercise in nostalgia -- an attempt to answer a 20th century question.  Mathematicians in the current century are properly interested in correspondences between (stable $\infty$-) categories of sheaves on two kinds of moduli spaces, or stacks:  one which parametrizes families, broadly interpreted, of representations of groups like $G(F)$, the other which parametrizes families of Langlands parameters with values in the $L$-group of $G$ over $F$.  This is the perspective of the article of Fargues and Scholze, who define an action of  category of the latter type (the {\it spectral category}) on one of the former type (the {\it automorphic category}), whose decategorification yields a map $\CL^{ss}$ that, they prove, enjoys some of the desired properties (see Theorem \ref{paramFS} for details).  The articles \cite{DHKM} and \cite{Zhu} address deep issues connected with the spectral category, at different degrees of magnification, especially when the coefficient field $C$ is replaced by a ring in mixed characteristic $(0,\ell)$, with $\ell \neq p$.   

The stack perspective provides one kind of explanation of the appearance of representations of component  groups of centralizers of Langlands parameters in determining the structure of (Vogan) $L$-packets.  All I can (or should) say here is that action of the whole centralizer group, and not just its group of connected components, brings to light the existence of
objects (coherent sheaves) on the moduli stack of Langlands parameters whose role in the classical (i.e. 20th century) local Langlands correspondence is not yet clear.  See however \cite[Remark I.10.3]{FS} for the return of a version of the classical local correspondence in the categorical context.

\subsection{Notation and conventions}  Let $F$ be a non-archimedean local field with integer ring $\CO_F$ and residue field
$k_F$ of characteristic $p$.  We follow Tate's conventions for the Weil group and Weil-Deligne group \cite{T79}.   Thus $W_F$ is the Weil group of $F$, given with a short exact sequence
$$1 \ra I_F \ra W_F \ra Frob_F^\ZZ \ra 1,$$
where $I_F$ is the inertia group and $Frob_F$ is the geometric Frobenius.  If $q = |k_F|$ 
there is a homomorphism
$$||\bullet||:  W_F \ra q^\ZZ$$
that is trivial on $I_F$ and sends any lift to $W_F$ of $Frob_F$ to $q^{-1}$.    The Weil-Deligne group is the group scheme defined 
as the semidirect product $\mathbb{G}_a \rtimes W_F$
where $wxw^{-1} = ||w|| x$ for $w \in W_F$.    A {\it Weil-Deligne representation} with coefficients in the field $E$ is a pair $(\rho,N)$ where
$\rho:  W_F \ra GL(V)$ is a representation with coefficients in $V$ and $N$ is a nilpotent operator on $V$ satisfying 
$\rho(w)N\rho(w)^{-1} = ||w||N$.  

We let $\G_F = Gal(F^{sep}/F)$ and let $W_F \ra \G_F$ denote the natural inclusion.

\part{The local Langlands conjectures}

\section{Tempered local parameters}\label{tempered}

Let $F$ and $G$ be as in the introduction.  We follow Vogan's approach in \cite{V93} to formulating the local Langlands conjecture, though we
mainly use Kaletha's conventions in \cite{Ka16}; at the end we will make
some comments about Kaletha's extension of Vogan's conjectures.  Let $C$ be an algebraically closed field of characteristic $0$, with a chosen square root $q^{1/2}$ of $q$, 
and let $\hG$ be the Langlands dual group of $G$ over $C$, $\hG(C)$ its group of $C$-valued points.   Let ${}^LG = \hG \rtimes W_F$ be the $L$-group of $G$
in the Weil group form.

The starting point for the local Langlands conjectures is the set $\F(G/F)$ of equivalence classes of admissible Langlands parameters ($L$-parameters, for short):
\begin{equation}\label{param}
\varphi:  WD_F \ra {}^LG(C).
\end{equation}
The homomorphism $\varphi$ has to respect the homomorphisms of both sides to $W_F$, to take $\mathbb{G}_a$ to a unipotent subgroup of $\hat{G}(C)$ normalized by the image $\varphi(W_F)$, and to take any lifting $w$ of $Frob_F$ to a semisimple pair $(s,w) \in \hG(C) \rtimes W_F$.
Two $\varphi$ are equivalent if they are conjugate by $\hG(C)$.   

\begin{conjecture}\label{LLC}  Let $\Pi(G/F)$ denote the set of irreducible admissible representations of $G(F)$ with coefficients in $C$.  

\begin{itemize}
\item[(a)]  There is a canonical parametrization
$$\CL = \CL_{G/F}:  \Pi(G/F) \ra \F(G/F).$$
\item[(b)]  For any $\varphi \in  \F(G/F)$, the {\bf $L$-packet} $\Pi_\varphi := \CL^{\varphi}$ is finite.
\item[(c)]  For any $\varphi \in  \F(G/F)$ the $L$-packet $\Pi_\varphi$ is non-empty.
\end{itemize}
\end{conjecture}
Parts (b) and (c) have been
stated separately because there are now canonical  {\it semisimple} parametrizations of representations of $G(F)$ by Weil group parameters $\varphi$ as in (a), defined
by means of algebraic geometry.  It is easy to see that property (b) holds for full parametrizations if and only if it holds for their semisimple parts; the  parametrizations defined by means of algebraic geometry are not known to satisfy this property.  Similarly, it is not known whether every semisimple parameter is reached by these parameters.

In what follows we assume $\varphi$ to be {\it tempered}.  To define these we associate to $\varphi$ a pair $(\rho,N)$, where $\rho:  W_F \ra \hat{G}$
and $N \in Lie(\hG)$ generates the tangent space to the image of $\varphi(\mathbb{G}_a)$.  To such a pair we can in turn associate a homomorphism
\begin{equation}\label{SL2} \varphi_{SL(2)}: W_F \times SL(2,C) \ra {}^LG(C)
\end{equation} 
in such a way that
$$\rho(w) = \varphi_{SL(2)}(w \times \begin{pmatrix} ||w||^{1/2} & 0 \\ 0 & ||w||^{-1/2}\end{pmatrix}),$$
which is meaningful because of our choice of $q^{1/2}$.  Without loss of generality we may assume $C = \CC$, and let $SU(2) \subset SL(2,\CC)$
be any compact real form.   Then $\varphi$ is tempered (resp. essentially tempered) if the image of the restriction of $\varphi_{SL(2)}$ to  $W_F \times SU(2)$ is bounded (resp. is bounded up to twist by a continuous homomorphism to the center of $\hG(\CC)$) after projection on the factor $\hat{G}(\CC)$.   

Let $\Pi^{\rm temp}(G/F) \subset \Pi(G/F)$ be the
subset of tempered representations.
We let $\F^{\rm temp}(G/F) \subset \F(G/F)$ denote the
subset of tempered Langlands parameters.  For any $\varphi \in \F^{\rm temp}(G/F)$,
$S_\varphi$ denote the centralizer  in $\hG(C)$ of the image of $\varphi$.   Let $G^*$ be the quasi-split inner form of $G$ and
let $Ad(G^*) = G^*/Z(G^*)$, where $Z(G^*)$ is the center of $G^*$.    Thus there
is an isomorphism $\xi:  G^* \ra G$ over $F^{sep}$,  , which gives rise to a $1$-cocycle $z_{ad}(\sigma) = \xi^{-1}\sigma(\xi) \in Z^1(\G,Ad(G^*))$. 
Following Kaletha (for $p$-adic fields) and Dillery (for local fields of
positive characteristic), we introduce the Galois gerbe $\Erig$ with the following property:  there is a canonical cohomology set $H^1_{\rm bas}(\Erig,G^*)$
and a surjective map
\begin{equation}\label{kot} H^1_{\rm bas}(\Erig,G^*) \ra H^1(\G,Ad(G^*)), \end{equation}
and we choose a lift of $z_{ad}$ to a $1$-cocycle $z \in Z^1_{\rm bas}(\Erig,G^*)$.    

Since $G$ is an inner form of $G^*$, we may identify ${}^LG = {}^LG^*$; the choice of identification depends on the choice of isomorphism $\xi$.
Let $S_\varphi^+$ and $Z([\hat{\bar{G}}]^+)$ denote the inverse images of $S_\varphi$ and $Z(\hG)^{\Gamma}$, respectively,
in the universal cover of $\hG$.
Kaletha's extension of
Vogan's formulation of the local Langlands correspondance is simplest to state for tempered representations:

\begin{conjecture}\label{Vogan}
Let $G^*$ be a quasi-split connected reductive group over $F$.
Then
\begin{itemize}

\item[(a)]  For each tempered Langlands parameter $\varphi \in \F^{\rm temp}(G^*/F)$, there is a set $\Pi_\varphi$ of isomorphism classes of quadruples $(G,\xi,z,\pi')$, where $(G,\xi,z)$ is a rigid inner twist of $G^*$ (the rigidification is supplied by the class $z$)
and $\pi'$ is a tempered irreducible representation of $G(F)$, and a commutative diagram

$$\begin{CD} 
\Pi_\varphi @>\iota_{W}>>   \operatorname{Irr}(\pi_{0}\left(S_{\varphi}^+\right))\\
@VVV   @VVV  \\ 
H^1_{\rm bas}(\Erig,G^*)    @>\pi_{0}>>  \left(Z([\hat{\bar{G}}]^+)\right)^{*}.
\end{CD}$$
Here the top map is a bijection, the bottom arrow is generalized
Tate-Nakayama duality, the right-hand arrow sends an irreducible representation of the centralizer to the character
of its center, and the left-hand arrow takes $(G,\xi,z,\pi')$ to the cohomology class of $z$.  
\item[(b)]  For each fixed triple $(G,\xi,z)$ and each $\varphi \in \F^{\rm temp}(G/F) = \F^{\rm temp}(G^*/F)$, the fiber in $\Pi_\varphi$ over the cohomology
class of $z$ (which determines the inner twist $G$) is exactly 
the fiber $\CL^{-1}(\varphi)$, where $\CL:  \Pi(G/F) \ra \F(G/F)$ is the parametrization of Conjecture \ref{LLC}.  In particular, the $L$-packet $\Pi_\varphi \subset
\Pi(G/F)$ attached to a tempered Langlands parameter is contained in 
$\Pi^{\rm temp}(G/F)$.
\item[(c)]  The top arrow $\iota_W$ is a bijection that takes the quadruple $(G^*,id,1,\pi^W)$ to the trivial representation, where $\pi^W$ is the unique
member of $\Pi_\varphi$ with a Whittaker model relative to a fixed preliminary choice of Whittaker datum $W$.  
  This provides a base point for $\Pi_\varphi$.
  \item[(d)]  In particular, the $L$-packet $\Pi_\varphi \subset \coprod_{G} \Pi(G/F)$, where $G$ runs over inner twists of $G^*$, is not empty and finite.  


\end{itemize}
\end{conjecture}

Part (c) of Conjecture \ref{Vogan} is a version of Shahidi's conjecture that every tempered $L$-packet of a quasi-split group contains a unique generic
member for each choice of Whittaker datum $W$; this choice provides a first normalization of the conjectural parametrization.   
For each inner twist $G$ of $G^*$ we let $\Pi_\varphi(G) \subset \Pi_\varphi$
be the set of quadruples with first entry $G$.   Point (c) guarantees that $\Pi_\varphi(G^*)$ is non-empty; more generally, the formulation of (c)
implies \cite[Lemma 5.7]{Ka16} that $\Pi_\varphi(G)$ is non-empty if and only if $\varphi$ is {\it relevant} to the inner form.

\subsection{Characterizing the local parametrization}\label{characterizing}
Conjecture \ref{Vogan} asserts the existence of a complete and canonical parametrization of tempered
representations of all inner twists of $G^*$.  It does not attach any meaning to the word ``canonical," and the word is in fact absent
from Vogan's original statement of the conjecture.  It was explained in the Introduction that no characterization of the map $\CL$  of Conjecture
\ref{LLC} has been proposed for general groups.  Nevertheless, one can list some properties, in addition to those already stated in the Conjecture, that $\CL$ is expected to satisfy.   First, we define the set $\F^{ss}(G/F)$ of {\it semisimple Langlands parameters} -- called {\it infinitesimal characters} in \cite{V93} --
to be the set of equivalence classes of admissible homomorphisms $W_F \ra {}^LG(C)$.  The {\it semisimplification map} $\F(G/F) \ra \F^{ss}(G/F)$ 
is given by restricting a homomorphism in $\F(G/F)$ to $W_F \subset WD_F$.  

\subsubsection{Basic properties}\label{basicp}

\begin{property}\label{p1}  When $G$ is a torus, the map $\CL$ is given by class field theory, as interpreted by Langlands in \cite{L97}.  In particular, when $G = GL(1)$ then
$\CL$ is the map induced from the reciprocity isomorphism
$$GL(1,F) = F^\times \isoarrow W_F^{\rm ab}.$$
\end{property}


\begin{property}\label{p2}  Let $\alpha:  F \ra F'$ be an isomorphism of fields.  Denote by  $\alpha_*$ the  identification of $\F(G/F)$ and $\F(G/F')$ corresponding
to the canonical isomorphism $WD_F \ra WD_{F'}$.  Identify $\Pi(G/F) \isoarrow \Pi(G/F')$ through the isomorphism $G(F) \isoarrow G(F')$.
Then the following diagram is commutative.
$$
\begin{CD}
\Pi(G/F) @>\CL_{G/F}>>  \F(G/F) \\
@VVV	   @VV\alpha_*V \\
\Pi(G/F') @>>\CL_{G/F'}>  \F(G/F')
\end{CD}
$$
\end{property}
The statement presupposes a choice of identification of $G/F$ with $G/F'$ that is used to identify the $L$-groups as well.  In practice we will only look
closely at the case where $F' = F$ or where $G$ is a split group, where the choice of identification is the obvious one.

\begin{property}\label{p3}  Let $\nu:  G \ra T$ be a homomorphism of algebraic groups over $F$, where $T$ is a torus.  Let $\nu^*:  {}^LT \ra {}^LG$ be the dual map.
Let $\pi \in \Pi(G/F)$, $\chi \in \Pi(T/F)$.  Then we have an equivalence
$$\CL(\pi\otimes \chi\circ\nu) \isoarrow \CL(\pi)\cdot\CL(\chi)\circ\nu^*.$$
\end{property}

\begin{property}\label{p4}  Letting $^{\vee}$ denote contragredient, we have
$$\sigma\circ\CL(\pi^{\vee}) \isoarrow (\sigma \circ\CL(\pi))^{\vee}$$
for any homomorphism $\sigma:  {}^LG \ra GL(N)$ of algebraic groups.
\end{property}

\begin{property}\label{p5}  Suppose $G = H_1 \times H_2$.  We let 
$$\otimes:  \Pi(H_1/F) \times \Pi(H_2/F) \ra \Pi(G/F)$$
 denote the  map on $L$-parameters induced by the inclusion
$$ {}^L(H_1 \times H_2) \ra {}^LH_1\times {}^LH_2$$
defined by the diagonal map $W_F \hookrightarrow W_F \times W_F$.   Then
$$\CL_{G/F} = \otimes(\CL_{H_1/F} \times \CL_{H_2/F})$$
\end{property}

\begin{property}\label{p6}  Let $r:  G \ra G'$ be a central isogeny and let $r^*:  {}^LG' \ra {}^LG$ denote the dual map.  Then if $\pi \in \Pi(G'/F)$,
$$\CL(\pi \circ r) = r^*\circ \CL(\pi).$$
If $\pi\circ r$ is reducible, the left hand side refers to any of the irreducible constituents; in particular, they all have the same parameter.
\end{property}

Note that Property \ref{p3} follows from Properties \ref{p5} and \ref{p6}.

\begin{property}\label{rEF}  Let $F$ be a finite separable extension of $E$ and $H = R_{F/E}G$ and let $\pi$ be an irreducible
representation of $H(E) = G(F)$.  Then parameters for $WD_F$ with values in ${}^LG = \hG\rtimes W_F$
can be naturally identified with parameters for $WD_E$ with values in ${}^LH = \hat{G}^{[F:E]} \rtimes W_E$.  With respect to this identification
we have
$$\CL_{H/E}(\pi) = \CL_{G/F}(\pi)$$
\end{property}

\begin{property}\label{p7}  Let $P \subset G$ be a parabolic subgroup rational over $F$, with Levi quotient $M$.  Let $i_M:  {}^LM \ra {}^LG$
be the inclusion as a Levi subgroup of the dual parabolic subgroup ${}^LP$.  Let $\pi \in \Pi(M/F)$.  Let $Ind_M^G$ denote normalized parabolic induction,
which is independent of the choice of $P$ containing $M$ (normalization using the chosen square root of $q$ in $C$).  Suppose $\tau$ is an irreducible subquotient of  $Ind_M^G(\pi)$.  Then
$$\CL^{ss}(\tau) \isoarrow i_M \circ \CL^{ss}(\pi).$$
\end{property}

Property \ref{p7} admits the following refinement.

\begin{property}\label{p7a}  Assume $G$ is quasi-split.  Let $K \subset G(F)$ be a special maximal compact subgroup, and let $H(G,K)$ denote the Hecke algebra of compactly supported  $K$-biinvariant functions on $G(F)$, with coefficients in $C$.  Suppose $\pi \in \Pi(G/F)$ is an irreducible spherical representation -- $\pi^K \neq \{0\}$.  Then $\CL(\pi) = \CL^{ss}(\pi)$ is the Satake parameter  $W_F \ra {}^LG(C)$ attached to the representation of $H(G,K)$ on  the $1$-dimensional space $\pi^K$.  

More precisely, let $B \subset G$ be a Borel subgroup, with Levi quotient $T$.  The hypothesis that $\pi$ be spherical implies that $\pi$ is a subquotient of
$Ind_T^G(\chi)$ for some character $\chi \in \Pi(T/F)$.  Then 
$$\CL(\pi) = \CL^{ss}(\pi) \isoarrow i_T \circ \CL(\chi).$$
 \end{property}

Complementary to Property \ref{p7} is

\begin{property}\label{discrete}  Suppose $\pi \in \Pi(G/F)$ belongs to the discrete series -- for example, if $\pi$ is supercuspidal.  Then
the quotient $S_{\CL(\pi)}/Z(\hG(C))^{\Gamma}$ is finite.  

In particular the image of $\CL(\pi)_{SL(2)}$ (notation as in \eqref{SL2}) is not contained in any proper parabolic
subgroup of $\hG \rtimes W_F$.  
\end{property}

Note that this {\it does not} assert that the image of $\CL^{ss}(\pi)$ is not contained in a proper parabolic.  

\begin{definition}\label{pure}  (a) The Weil-Deligne representation $(\rho,N)$ with values $GL(V)$, where $V$ is a $\CC$-vector space, is {\bf pure} if there is a 
homomorphism $t:  W_F/I_F \ra \CC^\times$, such that
\begin{itemize}
\item[(i)]  The eigenvalues of $\rho_t(Frob_F) := \rho(Frob_F)\cdot t(Frob_F)$  are all Weil $q$-numbers of integer weight for any Frobenius element $Frob_F \in W_F$.
\item[(ii)]  The subspace $W_aV \subset V$ of eigenvectors for $\rho_t(Frob_F)$ with eigenvalues of weight $\leq a$ is invariant
under $(\rho,N)$;
\item[(iii)]  Letting $gr_aV = W_aV/W_{a-1}V$, there is an integer $w$ such that, for all $i \geq 0$, the map
$$N:  gr_{w-i}V \ra gr_{w+i}V$$
is an isomorphism.
\end{itemize}

(b) The semisimple representation $\rho:  W_F \ra GL(V)$ is {\bf pure} if there is a homomorphism $t$ as above and an integer $w$ such that 
the eigenvalues of $\rho_t(Frob_F)$ are all $q$ numbers of the same weight $w$ for any Frobenius element $Frob_F \in W_F$.

(c)  Let $G$ be a connected reductive group over $F$.  The Weil-Deligne parameter $\CL:  WD_F \ra {}^LG(\CC)$ (resp. the semisimple parameter
$\CL^{ss}:  W_F \ra {}^LG(\CC)$) is {\bf pure} if for any representation (equivalently, for one faithful representation)
$\sigma:  {}^LG(\CC) \ra GL(V)$, the composite
$\sigma \circ \CL$ (resp. $\sigma \circ \CL^{ss}$) is pure in the sense of (a) and (b).

(d)  Let $\pi \in \Pi(G/F)$.  We say $\pi$ is {\it pure} if the semisimple parameter $\CL^{ss}(\pi)$ is pure in the sense of (b).
\end{definition}

We leave to the reader the formulation of the analogous definition with $\CC$ replaced by any algebraically closed $C$ of characteristic zero. 

\subsubsection{Automorphic properties}\label{autop}

The statement of Property \ref{p8}  requires the notion of a theory of {\it automorphic $L$-functions}.  What we have is a collection of
overlapping theories, each adapted to different classes of groups $G$.  Underlying principles have been identified but to date
there has been no exhaustive classification.  The following ad hoc definition will suffice for our purposes.  We let $C = \CC$.

\begin{defn}\label{autoL} Let $G$ be a  connected reductive algebraic group over a local field $F$.  Let $\sigma:  {}^LG \ra GL(N)$ be an
algebraic representation.  Let $\CA$ be a class of irreducible admissible representations of $G(F)$.
A {\bf theory of automorphic $L$-functions} for $G$ over $F$, $\CA$, and $\sigma$ consists of the following.
\begin{itemize}
\item[(a)]  For any representation $\pi \in \CA$, and any additive character
$\psi:  F \ra \CC^\times$, a meromorphic function
$L(s,\pi,\sigma)$ of $s \in \CC$ that is holomorphic in a right half-plane, and an entire function 
$s \mapsto \varepsilon(s,\pi,\sigma,\psi) \in \CC^\times.$  If $F$ is non-archimedean with residue field $k_F$
of cardinality $q$, then $L(s,\pi,\sigma)$ is of the form
$P(q^{-s})$ where $P$ is a polynomial of degree at most $\dim \sigma$ and constant term $1$, and $\varepsilon(s,\pi,\sigma,\psi)$ is a constant
multiplied by an exponential of the form $q^{\alpha s}$ for some complex number $\alpha$.  If $F$ is archimedean then $L(s,\pi,\sigma)$ and 
$\varepsilon(s,\pi,\sigma,\psi)$ can be written as products of the corresponding local factors ($\Gamma$-factors and $\varepsilon$-factors)
of representations in $\Pi(GL(1)/F)$.
\item[(b)] If $F$ is non-archimedean, with integer ring $\CO_F$,
and $\pi$ is spherical as in Property \ref{p7a}, with Satake parameter $\CL(\pi):  W_F \ra {}^LG(\CC)$, then 
$L(s,\pi,\sigma)$ is the Artin $L$-factor $L(s,\sigma\circ\CL(\pi))$.  Moreover, if $\psi$ is normalized so that $\ker \psi = \CO_F$,
then $\varepsilon(s,\pi,\sigma,\psi) = 1$.
\item[(c)]  Let $K$ be any global field with a place $v$ such that  $F \isoarrow K_v$, and let $\CG$ be a connected reductive group over $K$ with
$\CG(K_v) \isoarrow G(F)$.   There is a class $\CA_K$ of cuspidal automorphic representations of $\CG$ with the property
that any $\pi \in \CA$ can be realized, up to inertial equivalence {\bf see below}, as the local component at $v$ of some $\Pi \in \CA_K$.  It is assumed that for every place $w$ of $E$
and every local component $\Pi_w$ of $\Pi \in \CA_K$ there are local factors $L(s,\Pi_w,\sigma)$ and $\varepsilon(s,\Pi_w,\sigma,\bullet)$
satisfying (a) and (b), with the factors already defined when $w = v$.  
For any unitary cuspidal automorphic representation $\Pi \in \CA_K$ with $\Pi_v \isoarrow \pi$, the formal product
$$L(s,\Pi,\sigma) :=  \prod_w L(s,\Pi_w,\sigma)$$
converges absolutely in a right half-plane and extends to a meromorphic function of $\CC$ that satisfies the functional equation
$$L(s,\Pi,\sigma) = \varepsilon(s,\Pi,\sigma)\cdot L(1-s,\Pi^{\vee},\sigma).$$
Here
$$\varepsilon(s,\Pi,\sigma) := \prod_w \varepsilon(s,\Pi_w,\sigma,\psi_w)$$
if $\prod_w \psi_w$ defines a character of $\ad_K/K$.
\end{itemize}
Finally, 

(d) The theory is assumed to be {\bf intrinsic} to the group $G$, in the sense that it is  defined independently of an eventual functorial
transfer to $GL(n)$ with respect to $\sigma$.
\end{defn}

The introduction of the classes $\CA$ and $\CA_K$ is a roundabout way of avoiding reference to Whittaker models in the definition.  For some $G$ and $\sigma$
such theories only exist for globally generic $\Pi_E$, and for other $G$ and $\sigma$ the class $\CA$ contains all representations.  In any case one can only expect to realize a given $\pi \in \CA$ as a local component of a cuspidal automorphic representation up to inertial equivalence.  That is to say, if $\pi$ is a constituent of 
a parabolically induced representation $Ind_M^G \tau$ for some supercuspidal representation of $M$, then there is a character $\alpha:  M \ra \CC^\times$ such
that an irreducible constituent of $Ind_M^G \tau\otimes \alpha$ can be realized as a local component of a cuspidal automorphic representation of $\CG$.

\begin{property}\label{p8}  Suppose $\sigma:  {}^LG \ra GL(N)$ is an algebraic representation.   Suppose there is a theory of automorphic
$L$-functions for $G$ over $F$ and $\sigma$.  Then for any $\pi \in \Pi(G/F)$ and
$\psi:  F \ra \CC^\times$,
$$L(s,\pi,\sigma) = L(s,\sigma\circ\CL(\pi)); ~~ \varepsilon(s,\pi,\sigma,\psi) = \varepsilon(s,\sigma\circ\CL(\pi),\psi).$$
\end{property}

If $F$ is of positive characteristic it is probably possible to replace $\CC$ by a general $C$ in Property \ref{p8}.

The next Property refers freely to the notation and terminology of \cite{Ka22}.  If $\pi \in \Pi(G/F)$ we let $\Theta_\pi$ denote its
distribution character, which we view as a locally $L^1$-function on the set of regular elements of $G(F)$.    
Let $(H,\CH,s,\alpha)$ be an  endoscopic datum for the inner form $G$ of $G^*$, as in \cite{KS}.  Thus $H$ is a reductive group, $s \in \hat{G}(\CC)$
is a semisimple element, $\CH$ is a split extension of the dual group $\hat{H}$ of $H$ by $W_F$, and
$\xi: \CH \ra {}^LG(\CC)$ is an $L$-homomorphism that satisfies conditions \cite[2.1.4(a),(b)]{KS}; in particular, $\xi(\hat{H})$ is the identity component
of the centralizer of $s$.   

By Conjecture \ref{Vogan}, any member of $\Pi_\varphi$ defines a function on $S_\varphi^+$ by the formula
\begin{equation}\label{pairing}  (G,\xi,z,\pi') \mapsto  [s \mapsto Tr(\iota_W((G,\xi,z,\pi'))(\dot{s}))], ~ s \in S_\varphi^+;
\end{equation}
where $Tr$ denotes the trace and $\dot{s}$ is the image of $s$ in $\pi_0(S_\varphi^+)$.  For any $f \in C_c^\infty(G(F))$, we use
\eqref{pairing} to define
$$O^s_\varphi(f) = \sum_{\pi' \in \Pi_\varphi(G)} Tr(\iota_W(G,\xi,z,\pi')(\dot{s}) \Theta_{\pi'}.$$
We use the notation $SO$ to denote $O^1$, i.e. $O^s$ with $s$ the identity.

We state the property under the simplifying assumption that $\CH$ is in fact ${}^LH(\CC)$.

\begin{property}\label{p9}  
Let $\varphi_H \in \F(H/F)$ be a tempered Langlands parameter; let $\varphi = \alpha\circ\varphi_H$ be
the induced parameter of $G$.    

Then the characters of the elements of $\Pi_{\varphi_H}(H) \subset \Pi(H/F)$ and of $\Pi_{\varphi}(G) \subset \Pi(G/F)$
satisfy the following endoscopic character identity:
\begin{equation}\label{endoscopicidentity}  SO_{\varphi_H}(f^H) = O^s_{\varphi}(f)
\end{equation}
Here $f \in C_c^\infty(G(F))$ is a test function and $f^H$ is its transfer with respect to an appropriate transfer factor (defined in \cite{Ka16})
\end{property}

\begin{remark}  Bertoloni Meli and Youcis have given a  characterization of local Langlands correspondences for supercuspidal
representations in certain cases, based on a version of Property \ref{p9} \cite{BMY}.
\end{remark}

Kaletha informs me that there is a partial extension of this property to the case of twisted endoscopy, as treated in \cite{KS},
at least for $p$-adic fields.
We will be particularly interested in the following special case of twisted endoscopy.

\begin{property}\label{p10}  Suppose $H$ is quasisplit, $G = R_{F'/F}H_{F'}$ where $F'/F$ is a cyclic extension of prime order $\ell$,
and $\theta$ is the automorphism of $G$ induced by a generator (also denoted $\theta$) of $Gal(F'/F)$.  In particular, $\CH = {}^LH$,
${}^LG = \hat{H}^\ell \rtimes W_F$, and $\alpha$ corresponds to the diagonal inclusion of $\hat{H}$ in $\hat{H}^\ell$.  

Let $\varphi_H \in \F(H/F)$ be a tempered Langlands parameter and let $\varphi = \alpha \circ\varphi_H$ be the corresponding parameter of $G$.  Then
for all regular semisimple elements $\delta$ of $G(F)$, we have
$$\sum_{\pi' \in \Pi_{\varphi_H}(H)} a(\pi') \Theta_{\pi'}(N_\theta(\delta)) = c\sum_{\pi \in \Pi_\varphi(G)} b(\pi)\Theta_\pi(\delta)$$
for complex constants $a(\pi')$, $b(\pi)$, and for a constant $c$ that can be normalized to equal $1$ with additional choices.  Here $N_\theta$ denotes the norm map on regular semisimple conjugacy classes.
\end{property}

\begin{remark}  Because there is no general statement of Property \ref{p10} in the literature, the referee advised me to introduce the constants $a(\pi')$ and $b(\pi)$ in
the above formula, buit pointed out that, in all  cases  studied thus far, the $a(\pi')$ and $b(\pi)$ all turn out to equal $1$.
\end{remark}

Let $G$ be a connected reductive group over the global field $K$.  Let $S$ be a finite set of places of $K$ and for each $v \in S$ let 
$\Pi_v$ be an admissible irreducible representation of $G(K_v)$.  In practice, $S$ will be empty if $K$ is a function field
and will consist of the set of archimedean places if $K$ is a number field, but we lose nothing by considering the more general
situation.  Let $\CA_0(K,\Pi_S)$ be the set of cuspidal automorphic representations $\Pi$ of $G(\ad_K)$ with $\Pi_S \isoarrow \otimes_v \Pi_v$.

Let $\F(G/K)$ be the set of (equivalence classes of) compatible families of homomorphisms  
$$\varphi_\ell:  Gal(K^{sep}/K) \ra  {}^CG(\Qlb)$$
that are unramified outside a finite
set of places of $K$, where ${}^CG$ is the $C$-group introduced in \cite{BG} if $K$ is a number field and is ${}^LG$ if $K$ is a function field;  $\ell$ runs over all primes different from the characteristic of $K$.    Say $(\varphi_\ell)$ is semisimple if each $\varphi_\ell$ has the following property:  
if (the projection on the identity component of) ${}^CG(\Qlb)$ of $\varphi_\ell(Gal(K^{sep}/K)$ is contained in a proper parabolic subgroup then it is
contained in a Levi factor.  Let
 $\F(G/K)^{ss} \subset \F(G/K)$ denote the subset of semisimple representations.

\begin{property}\label{p11}  Suppose there is a global semisimple parametrization of the automorphic representations $\CA_0(K,\Pi_S)$
$$\CL^{ss}_K: \CA_0(K,\Pi_S) \ra \F(G/K)^{ss}$$
that is normalized so that, for each $\Pi \in \CA_0(K,\Pi_S)$ and all $w \notin S$ such that $\Pi_w$ is spherical, the restriction $\CL^{ss}(\Pi)_w$ of $\CL^{ss}(\Pi)$ to a decomposition group $\G_w$ at $w$ is unramified and $G$ is quasi split, the {\bf semisimplification} of $\CL^{ss}(\Pi)_w$ is the Satake parameter
 of $\Pi_w$.

Then for all places $v \notin S$, the semisimplification of $\CL^{ss}(\Pi)_v$ equals $\CL^{ss}(\Pi_v)$.
\end{property}

Global semisimple parametrizations are constructed in complete generality in \cite{Laf18} when $K$ is a function field and $S = \emptyset$.  The construction has been extended to all automorphic representations by C. Xue \cite{CX}.  When $K$ is a number field,
semisimple parametrizations for certain classes of $G$ and $\Pi_S$ have been constructed by means of Shimura varieties.  These will be discussed below.

\begin{property}\label{p12}[Independence of $\ell$]  Let $C$ run over the fields $\Qlb$, as $\ell$ varies among
primes different from $p$.  For each $\ell$ there is a parametrization $\CL^{ss}_\ell$, with the following property.
Let $E$ be a number field and assume $\pi \in \Pi(G/F)$ has a model over $E$.  
Then for any $\sigma:  {}^LG \ra GL(N)$ and for every $g \in W_F$, the characteristic polynomial of 
$\sigma \circ \CL^{ss}_\ell(\pi)(g)$ has coefficients in $E$, independent of $\ell$.
\end{property}


\section{The case of $GL(n)$}

With regard to the local Langlands correspondence, the situation for $GL(n)$ is satisfying for a number of reasons.  Let $F$ be any nonarchimedean local field,
and let $G = GL(n)$.  The following theorem has been proved in several ways, which are compared in \S \ref{comparedproofs}.
\begin{thm}\label{LLCGLN}  Conjecture \ref{LLC} is valid for $G/F$.  Moreover, the collection of local parametrizations $\CL = \CL_{GL(n)/F}$, as $n$ varies,
satisfies all the relevant Properties in \S \ref{basicp} and \S \ref{autop}.  

More precisely, $\CL_{GL(n)/F}$satisfies 
\begin{itemize}
\item[(1)] Property \ref{p1} for $n = 1$;
\item[(2)] Properies \ref{p2}, \ref{p4}, \ref{p7}, \ref{p7a}, \ref{rEF}, and \ref{discrete}, as written;
\item[(3)]  Property \ref{p3} for the homomorphism $\nu = \det:  GL(n) \ra GL(1)$;
\item[(4)] Property \ref{p8} for representations of $G = GL(n) \times GL(m)$, where the theory of Definition \ref{autoL} is the {\bf Rankin-Selberg theory} --
either that of Jacquet-Piatetski-Shapiro-Rallis or Shahidi \cite{JPSS,Sh83} -- with $\sigma:  G \ra GL(nm)$ is the tensor product representation, and $\CA$ is the class of {\bf generic} representations.  
Property \ref{p8} also holds for several other theories of automorphic $L$-functions with $G = GL(n)$ (see Remark \ref{autoLGLN}, below).
\item[(5)] Property \ref{p9} will be discussed in connection with classical groups.  On the other hand, Property \ref{p10} is valid for $GL(n)/F$.
\item[(6)]  When $F$ is of positive characteristic, property \ref{p11} is valid  and $\CL^{ss}$ is defined in terms of the cohomology of moduli spaces of shtukas,
either in \cite{LRS93} (for compact moduli spaces) or in \cite{Laf02} (for general cuspidal automorphic representations; or indirectly as in \cite{Laf18}.
\item[(7)]  When $F$ is a $p$-adic field, property \ref{p11} is valid when the global $\CL^{ss}$ is defined on Shimura varieties attached to unitary (similitude) groups and their inner twists.
\end{itemize}
\end{thm}

\begin{remark}\label{autoLGLN}   The book of Godement-Jacquet \cite{GJ72} defines a theory of automorphic $L$-functions, with $\sigma$ the standard
representation of $GL(n)$ to itself, with $\CA = \Pi(GL(n)/F)$.  The parametrization $\CL_{GL(n)/F}$ satisfies Property \ref{p8} for this theory, and for theories defined
when $\sigma$ is either the exterior square or the symmetric square, or when $n = 2$ and $\sigma$ is the symmetric cube.  See \cite{CST, HeLo21, He21} for details.

In equal characteristic Property \ref{rEF} is a simple consequence of the global Langlands correspondence for $GL(n)$ over function fields.  
As far as I know, Property \ref{rEF} has not been verified in print for $p$-adic fields, but it follows easily from compatibility with $L$ and $\varepsilon$ factors for $GL(n)$.
\end{remark}

Henniart proved in \cite{He93} that the properties \ref{p1}, \ref{p3}, and \ref{p8} for the Rankin-Selberg local factors characterize the collection of 
$\CL_{GL(n)/F}$ uniquely, for all $n$.  The other properties come as a bonus; however, the known constructions of $\CL_{GL(n)/F}$, for $n > 3$,
are all based on Property \ref{p11} and the proof that they give bijections make strong use of Property \ref{p10} and a variant, not stated here, for
automorphic induction.

\subsection{The numerical correspondence and Henniart's splitting theorem}  Let $\Phi^0(GL(n)/F) \subset \Phi(GL(n)/F)$ denote the subset of
irreducible parameters -- in other words, irreducible $n$-dimensional representations of $W_F$.  Let $\Pi^0(GL(n)/F) \subset \Pi(GL(n)/F)$ denote
the subset of supercuspidal representations.
The proofs of Theorem \ref{LLCGLN} in \cite{LRS93,Laf02,Laf18} (in positive characteristic)
and in \cite{HT01,He00,Sch13} (for $p$-adic fields) all depend crucially\footnote{However, upon closer examination one sees that these results
are superfluous for the proofs in \cite{Laf02,Laf18}.}
on at least one of the two closely related results first proved by Henniart in \cite{He88,He90}:

\begin{thm}\cite{He88}[Henniart's numerical correspondence]\label{num}   Let 
$$\CL: \Pi^0(GL(n)/F) \ra \Phi^0(GL(n)/F)$$ 
be an injective (resp. surjective) map.  Suppose $\CL$ preserves conductors (a weak form of Property \ref{autoL} for $\varepsilon$ factors) and is compatible with
twists by characters of finite order (a weak form of Property \ref{p3}).  

Then $\CL$ is bijective.
\end{thm}

The proof is based on an idea that goes back to Tunnell's proof \cite{Tu78} of most cases of the local Langlands correspondence for $GL(2)$.  The two sides of the map $\CL$ are partitioned according to conductors.  Each equivalence class is the union of finitely many orbits under twisting by characters.  The number of orbits of fixed conductor
on the left-hand side can be counted by using the Jacquet-Langlands correspondence to reduce the problem to counting irreducible representations of finite quotients
of the multiplicative group of a division algebra.  When $F$ is of positive characteristic,  Henniart counts the number of orbits on the right-hand side with the help of
Laumon's Fourier transform; for $F$ $p$-adic there is an intricate argument using the Deligne-Kazhdan method of close local fields (see \S \ref{CLF}, below).

\begin{thm}\cite{He90}[Splitting theorem]\label{split}  Let $F_0$ be a non-archimedean local field, $n$ a positive integer, and $\pi_0$ a supercuspidal representation of $G(F_0)$.  There is a finite sequence of cyclic extensions $F_0 \subset F_1 \subset \dots \subset F_r$ such that, if we define $\pi_i$ by induction as the representation of $G(F_i)$ obtained as the cyclic base change from $F_{i-1}$ of $\pi_{i-1}$, then $\pi_r$ is not supercuspidal.
\end{thm}  

The statement presupposes the existence of cyclic base change, which is established by Arthur and Clozel in \cite{AC}.   Henniart's proof in \cite{He90} is a two-page
deduction from his construction in \cite{He88} of numerical correspondences $\CL$ satisfying the properties of Theorem \ref{num}.   The point is that one can find
a sequence of $F_i$ as in the statement of the theorem so that $\CL(\pi_r)$ is unramified.  Although at the time it was not known
that the among the $\CL$ he constructed there was one that satisfies Property \ref{p8}, his correspondences were sufficiently well-behaved with respect to base change
that he could deduce that if $\CL(\pi_r)$ is unramified then it cannot be supercuspidal.

\subsection{Comparing proofs of the local correspondence for $GL(n)$}\label{comparedproofs} 

It was agreed from the beginning that a local parametrization 
$$\CL:  \Pi(GL(n)/F) \ra \Phi(GL(n)/F)$$ deserved to be called the local Langlands correspondence
only if it was a {\bf bijection} that satisfied Property \ref{p8} for the Godement-Jacquet theory of local $L$ and $\varepsilon$ factors.  This was the form in which the conjecture was proved for $GL(2)$ in most cases by Tunnell \cite{Tu78}, then in all cases by Kutzko \cite{Ku}.  The two proofs had a significant difference:  Tunnell's was based on the counting argument described above, whereas Kutzko's was the first to be based on classification by types, as discussed in \S \ref{typesGL} below.  
Henniart's proof for $GL(3)$ was also based on types.

The goal posts had been moved by the time Laumon, Rapoport, and Stuhler proved the local Langlands conjecture for $GL(n)$ over local fields of positive characteristic \cite{LRS93}.  Now a local Langlands correspondence had to satisfy Properties \ref{p1}, \ref{p3}, and \ref{p8} for the Rankin-Selberg local factors; as mentioned above, Henniart had proved that these properties sufficed to characterize a correspondence uniquely.  He had in the meantime used Theorem \ref{num} to construct numerical correspondences -- bijections that preserved conductors.  But his method did not prove Property \ref{p3} , and in fact could not, because it gave a family of correspondences that differed by uncontrolled unramified twists, rather than a single correspondence.

The first canonical bijection $\CL:  \Pi^0(GL(n)/F) \ra \Phi^0(GL(n)/F)$, when $F$ is a $p$-adic field was constructed in \cite{H97}.  It was initially constructed as a map, $\CL:  \Pi^0(GL(n)/F) \ra \Phi(GL(n)/F)$ arising from a realization in cohomology, specifically the cohomology of an appropriate $p$-adic symmetric space.  To prove that it was a bijection with the irreducible subset $\Phi^0 \subset \Phi$ followed from a theorem proved in a paper of Bushnell, Henniart, and Kutzko \cite{BHK}:  any 
such $\CL$ is a bijection provided it satisfies Properties \ref{p1}, \ref{p2}, \ref{p3}, \ref{p4}, \ref{p10}, and two other properties specific to $GL(n)$:  the central character of $\pi$ corresponds by class field theory to $\det \CL(\pi)$, and $\CL$ is compatible with cyclic automorphic induction.  This proof relied on Henniart's Theorem \ref{split}, and thus  indirectly on Theorem \ref{num}.

However, the methods of \cite{H97} and \cite{BHK} were not sufficient to prove Property \ref{p8} for $L$ and $\varepsilon$ factors of pairs, although \cite{BHK} did show that it preserved {\it conductors} of pairs.  Thus it could not yet be called {\it the} local Langlands correspondence.

All known proofs of \ref{p8} for $GL(n)\times GL(m)$ depend on a global construction -- in other words, on local-global compatibility as in Property \ref{p11}, for a sufficiently general class of cuspidal automorphic representations.  The class had to consist of $\Pi$ with two properties:  the existence of global
parameters $\CL(\Pi)$, as in \ref{p11}, and the existence of global functional equations for the $L$-functions attached to the Galois parameter $\CL(\Pi)$ that are independent of the equations known for $\Pi$ itself.  In positive characteristic the second property was a special case of a theorem of Grothendieck that applied to all cuspidal $\Pi$, and the proof in \cite{LRS93}, like the later proofs in \cite{Laf02} and \cite{Laf18}, consisted in constructing the global parameters $\CL(\Pi)$.
For number fields a special class of $\Pi$ was constructed in \cite{H98} that satisfied the necessary properties and was  used in the proofs in \cite{HT01,He00,Sch13}.   More recently,
Daniel Li-Huerta has given a new proof in positive characteristic \cite{LH21}, modeled on the proof in \cite{Sch13} for $p$-adic fields.

\subsection{Classification by types}\label{typesGL}

A second classification of $\Pi^0(GL(n)/F)$ is given in terms of representations of compact open subgroups.  Let $G$ be a general reductive group over $F$, and let $Z_G \subset G$ denote the center.   Let $U \subset G(F)$ be an open subgroup that contains $Z_G(F)$ and is compact modulo $Z_G(F)$.  Let $\lambda:  U \ra GL(V)$ be a continuous
finite-dimensional representation with coefficients in $C$. Let $C_c(G(F),V)$ denote the space of continuous (locally constant) functions on $G(F)$, with values in $V$,
that are compactly supported modulo $Z_G(F)$.   The compact induction $c-Ind_U^{G(F)} \lambda$ is the space
\begin{equation}\label{cind}  \{f \in C_c(G(F),V) | f(ug) = \lambda(u)f(g) ~\forall g \in G(F), u \in U \}.
\end{equation}

It is a general fact that if $c-Ind_U^{G(F)} \lambda$ is irreducible then it is supercuspidal.  For $G = GL(n)$ the book of Bushnell and Kutzko \cite{BK} provides the following classification
of $\Pi^0(GL(n)/F)$ as compactly induced representations:

\begin{thm}[\cite{BK}]\label{BKtypes}  Every supercuspidal representation of $GL(n,F)$ is isomorphic to one induced from an irreducible representation of an open subgroup that is compact
modulo $Z_{GL(n)}(F)$.  More precisely, let $\pi \in \Pi^0(GL(n)/F)$.  Then there is a pair $(\bJ,\Lambda)$, with $\bJ$ compact modulo $Z_{GL(n)}(F)$
and $\Lambda:  \bJ \ra GL(n,V)$ an irreducible continuous finite-dimensional representation, with the following properties: 
\begin{itemize}
\item[(i)]  Let $\pi'$ be an irreducible representation $GL(n,F)$ and suppose  $\Lambda$ occurs in the restriction of $\pi'$ to $\bJ$.  Then $\pi' \isoarrow \pi$;
\item[(ii)] The pair $(\bJ,\Lambda)$ is unique up to conjugation;
\item[(iii)]   $\pi \isoarrow c-Ind_{\bJ}^{GL(n,F)} \Lambda$.
\end{itemize}
\end{thm}

Property (i) asserts that the pair $(\bJ,\Lambda)$ is a {\it type} (more precisely an {\it extended type}) for $\pi$.   In \cite{BK99} Bushnell and Kutzko described a more general theory of types for irreducible admissible representations of $GL(n,F)$; the representations $\pi$ that contain such types correspond to a connected component of the Bernstein center of $GL(n,F)$.  

More will be said about types  in \S \ref{types} below.  For now I quote what I wrote in my 2002 ICM talk, allowing for the change in notation:

{\it The outstanding open problem concerning the local Langlands correspondence
is undoubtedly} 
\begin{itemize}
\item Define $\CL(\pi)$ directly in terms of the type $(\bJ,\Lambda)$ (and vice versa)
\item  Show directly that the definition of $\CL(\pi)$ in terms of $(\bJ,\Lambda)$ has the expected properties of a local Langlands correspondence.
\end{itemize}

Although there has been some significant progress, especially when the residue characteristic $p$ is prime to $n$ \cite{BH10,Tam16}, this problem is still open in general nearly 20 years later.

\subsubsection{Inner forms}  Let $G$ be an inner form of $GL(n)$ over $F$.  Thus $G$ can be identified with $GL(a,D)$ where $D$ is a central division algebra over $F$
of dimension $b^2$ and $ab = n$.  The generalized Jacquet-Langlands correspondence includes a bijection
\begin{equation}\label{JL}  JL:  \Pi^{0,0}(G/F) \isoarrow \Pi^0(GL(n)/F)
\end{equation}
where $\Pi^{0,0}(G/F) \subset \Pi^0(G/F)$ is an appropriate subset of the cuspidal representations,
which is characterized by an identity of characters, as follows.   If $\gamma \in G(F)$ and $\delta \in GL(n,F)$ are semisimple elements (or conjugacy classes), we write
$\gamma \leftrightarrow \delta$ if their characteristic polynomials are equal over $\bar{F}$.  We denote the distribution character of $\pi \in \Pi^0(?/F)$ by
$\chi_\pi$.  Now if $\pi \in \Pi^0(G/F)$ then for any elliptic regular element 
$\gamma \in G(F) = GL(a,D)$ with $\gamma \leftrightarrow \delta$ as above we have
\begin{equation}\label{charJL}  \chi_\pi(\gamma) = (-1)^{a-1} \chi_{JL(\pi)}(\delta).
\end{equation}

The analogue of Theorem \ref{BKtypes} for $GL(a,D)$ was developed in a series of articles by S\'echerre and Stevens, joined for one article by Broussous, and
completed in \cite{SS}.

\section{Other groups}

\subsection{Classical groups}\label{classicalsection}

Arthur's book \cite{A} reduces the local and global Langlands correspondences for split symplectic and special orthogonal groups, when $F$ is a $p$-adic field,
 to the correspondences for general linear groups.  As a consequence he obtains what should be a definitive version of the local correspondence for tempered representations of these classical groups
 over $p$-adic fields.  Subsequent work by Mok and Kaletha, M\'inguez, Shin, and White \cite{Mok,KMSW}  obtains analogous correspondences for unitary groups.  A complete theory  for classical groups seems still to be conditional on results in harmonic analysis whose proofs have not yet appeared.
 
For even $n$ the correspondence has its
most natural statement for $O(n)$ rather than $SO(n)$.  When $G$ is the split form of $SO(n)$ with $n$ even, and $\pi, \pi' \in \Pi^{\rm temp}(G/F)$, we write
$\pi \sim \pi'$ if they are equivalent under the outer action of $O(n,F)$ on $SO(n,F)$; if $G$ is another group then the relation $\sim$ is just equality.
 Here is the statement of Arthur's version of the local correspondence, in the language of \S \ref{tempered}.
 
 \begin{thm}[Arthur]\label{arthurthm}  Let $G = G^*$ be a split symplectic or quasi split special orthogonal group, 
 over the $p$-adic field $F$.  Let $\varphi \in \F^{\rm temp}(G/F)$.
 Then there is a finite subset $\Pi_\varphi \subset \Pi^{\rm temp}(G/F)/\sim$ that is canonically in bijection with the group of characters of the (abelian) group
 $\pi_0(S_\varphi)/Z(\hat{G}).$  The corresponding map
 $$\CL:  \Pi^{\rm temp}(G/F)/\sim \ra \F^{\rm temp}(G/F)/\sim$$
 satisfies Properties \ref{p2}, \ref{p7}, \ref{p7a}, \ref{rEF}, \ref{discrete}, and \ref{p9}.
 \end{thm}
 
 The bijection with the characters of  $\pi_0(S_\varphi^+)/Z(\hat{G})$ is just the part of Condition (b) of Conjecture \ref{Vogan} that corresponds to the quasi-split inner form.  
 
 Theorem \ref{arthurthm} is proved by comparing the stable trace formula for $G$ with the twisted trace formula for the outer automorphism of $GL(N)$, where $N$ depends on $G$.  Thus by construction the parametrization for $G$ is compatible with twisted endoscopic transfer from $G$ to $GL(N)$.  Properties \ref{p2} and \ref{rEF} for $G$ are not stated explicitly anywhere, as far as I know, but they follow immediately from the corresponding properties for $GL(N)$.   The deduction of the local correspondence for $G$
 from results for $GL(N)$ is one part of an elaborate multiple inductive argument, based on comparison of local as well as global
 trace formulas, that stretches over the entire length of \cite{A}.  One reading of the proof would  place the starting point for the induction in \S\S 6.5-6.7 of \cite{A}:  the orthogonality relations for (elliptic) characters are needed in order to characterize the structure of $L$-packets, and specifically to prove that the individual representations appear in the character with multiplicity one.  This idea is the inspiration for the strategy that will be outlined in Part \ref{Strategy}.
 
Property \ref{p8} can be formulated in terms of the theory of automorphic $L$-functions obtained by means of the {\it doubling method} of Piatetski-Shapiro
and Rallis.  It seems to me that this version of Property \ref{p8} can be deduced  by comparing the functional equation
for the doubling $L$-function of $G$ with the standard functional equation for $GL(N)$, using the results of Rallis and Soudry in \cite{RS}, but I have been unable to find
a proof in the literature.  Ongoing work of Cai, Friedberg, Ginzburg, and Kaplan, starting with \cite{CFGK}, should similarly give the  version of Property \ref{p8} for the
L-functions of $G \times GL(n)$, but I haven't seen a definite statement for this case either.   Property \ref{p11} (local-global compatibility) for Arthur's correspondence,
where the global parametrization for $G$ is given by the cohomology of certain Shimura varieties of abelian type, seems to be open in most cases at the time of writing.


\subsection{$\mathbf{G_2}$}

For any quasi-split group $G$ over $F$, let $\Pi^0_g(G/F)$ denote the set of (equivalence classes of) generic supercuspidal representations of $G(F)$.
Ten years before the publication of Arthur's book on classical groups, Jiang and Soudry \cite{JSo} obtained a parametrization of $\Pi^0_g(G/F)$ in terms of Galois parameters,
when $G$ is a split odd orthogonal group.  
\begin{thm}\label{JSopaper}  For any positive integer $n$, there is a parametrization of {\bf generic} supercuspidal representations of $SO(2n+1,F)$
$$\CL_g:  \Pi^0_g(G/F) \isoarrow \Phi^0(G/F)$$
that satisfies Property \ref{p8} for an appropriate theory of automorphic $L$-functions.
\end{thm}
A proof of the analogous result for even orthogonal groups and symplectic groups is sketched in \cite{JSo2}.  The structure of the proof is discussed below, after the statement of an analogous result for $G = G_2$.

By the time this article is published it is likely that Gan and Savin will have completed their program to define a local Langlands correspondence for inner forms of $G_2$
satisfying many of the properties in \S \ref{characterizing}.  Their construction is based on an analysis of the exceptional theta correspondences for the dual reductive pairs
$(G_2,PGL(3))$ and $(G_2,PSp(6))$ in exceptional groups of type $E_6$ and $E_7$, respectively.  The analysis has been carried out in a long list of papers by (in more or less
chronological order) Ginzburg, Rallis, Soudry, Li, Magaard, Gross, Jiang, Woodbury, and Weissman, as well as Gan and Savin.  Parameters with values in the dual group $\hat{G}$ of $G_2$, which is $G_2$ itself, can be studied by means of the $7$-dimensional representation of $\hat{G}$, which takes values in $SO(7)$.  Parameters for $G_2$ are then matched with the known parameters for $PGL(3)$ and $PGSp(6)$ with values in $SO(7)$ and the relations among these parameters are compared with those constructed 
by means of the exceptional theta correspondences.  At the time of writing there are still a few issues to be sorted out over $2$-adic fields.

The article \cite{HKT} takes the results on the exceptional theta correspondences as its starting point, but applies principles developed in connection with the deformation theory of global Galois representations.  The main result is the following.  We take $G_2$ to designate the split form of the group.
As above, let $\Pi^0_g(G_2/F)$ denote the set of (equivalence classes of) generic supercuspidal representations of $G_2(F)$,
and let $\Phi^0(G_2/F)$ denote the subset of $\Phi(G_2/F)$ consisting  of (equivalence classes of) irreducible local parameters.   The following theorem is
the analogue for $G_2$ of Theorem \ref{JSopaper}.

\begin{thm}\label{HKTthm}  There is a natural bijection
$$\CL_g:  \Pi^0_g(G_2/F) \isoarrow \Phi^0(G_2/F).$$
\end{thm}

Since the first step in the construction involves the exceptional theta correspondence between $G_2$ and $PGSp(6)$, one could with equal justice 
call $\CL_g$ an {\it unnatural} bijection.  The proof is analogous to that of Theorem \ref{JSopaper} in that it relies essentially on the method of automorphic
descent.  Where Jiang and Soudry quoted the automorphic descent from general linear groups to classical groups as developed by Ginzburg, Rallis, and Soudry \cite{GRS},  Theorem \ref{HKTthm} uses a result of Hundley and Liu that constructs generic automorphic representations of $G_2$ using Fourier coefficients on
$E_7$ of Eisenstein series obtained from $GL(7)$.   As in \cite{GRS}, the automorphic representation of $GL(7)$ descends to a generic automorphic representation
of $G_2$ precisely when a certain $L$-function of the general linear group -- in this case the exterior cube $L$-function -- has a pole.   The proof in \cite{HKT}
differs from previous proofs in that it makes use of automorphic lifting theorems and potential automorphy to realize a given parameter in $\Phi^0(G_2/F)$.  A more general
application of this strategy is proposed in Part \ref{Strategy}.

 Most of the properties listed in \S \ref{basicp} are irrelevant to $G_2$.  However, the local exceptional theta lifting 
is compatible with the global theta lifting of automorphic representations, and the local unramified lift to $PGSp(6)$, composed with the pullback to $Sp(6)$ 
and then unramified functoriality from $Sp(6)$ to $GL(7)$ was computed to correspond to Langlands functoriality for the $L$-homomorphism
$$\hat{G}_2 = G_2 \overset{r_7}\to GL(7) = \widehat{GL}(7),$$
where $r_7$ is the (unique) $7$-dimensional representation of $G_2$.  It follows from a theorem of Griess \cite[Theorem 1]{Gr} that the composition
$$r_7 \circ \CL_g:  \Pi^0_g(G_2/F) \ra GL(7)$$
determines $\CL_g$ up to $G_2$-conjugation.   Thus by Chebotarev density and a theorem of Chenevier  \cite{Ch} the  parametrization is uniquely determined.

In particular, one verifies easily, by comparison with the local Langlands correspondence for $GL(7)$, that the bijection $\CL_g$ satisfies Properties \ref{p2}, \ref{p4}, \ref{rEF}, \ref{discrete}, \ref{p11} and \ref{p12}, insofar as they apply to generic supercuspidal representations.  There are also several constructions of automorphic $L$-functions for $G_2$, but as far as I know they have not been developed to the point of defining local $\gamma$-factors for ramified representations, so one can't yet assert that
$\CL_g$ verifies property \ref{p8}.

\section{Supercuspidals by types}\label{types}

\subsection{Results of Yu, Kaletha, and Fintzen}

In this section we will assume $G$ to be semi-simple, to simplify the statements.  It is generally believed that every supercuspidal representation of $G(F)$
can be constructed by compact induction from an appropriate type $(\bJ,\Lambda)$, as in Theorem \ref{BKtypes}.  In \cite{Yu01} J.-K. Yu introduced a
rather elaborate construction of types $(\bJ,\Lambda)$, for any reductive group that splits over a tamely ramified extension of $F$, and proved that $c-Ind_{\bJ}^{G(F)} \Lambda$ is irreducible, hence
supercuspidal.  Although the proof contained no error, it was based on a statement that had been published with a misprint, and was false.  A correct proof of
Yu's theorem was given more recently by Fintzen \cite{Fi19}, in connection with her proof of the fundamental theorem \cite{Fi21} that we have already mentioned several
times.

\begin{thm}\label{fintzenthm}  Suppose $G$ splits over a tame extension of $F$ and $p$ does not divide the order of the Weyl group $W(G)$ of $G$.  Then every irreducible supercuspidal representation
of $G(F)$ arises from Yu's general construction.  In particular, every irreducible supercuspidal representation of $G(F)$ is compactly induced from an 
irreducible representation of a compact open subgroup of $G(F)$.
\end{thm}

When applied to $GL(n)$, Yu's construction recovers precisely the supercuspidals whose corresponding Langlands parameters are induced from characters of
tame extensions of $F$, and indeed Yu's types can be identified with those of Bushnell and Kutzko \cite{MY}.  In particular, Yu's construction recovers all supercuspidals when $n$ is prime to $p$.  In general it is expected that the Langlands parameters of Yu's supercuspidals
have the property that the image of wild inertia is contained in a maximal torus.  A precise conjecture along these lines is due to Kaletha.  In order to formulate his conjecture,
Kaletha first introduces a more manageable parametrization of a subclass of Yu's supercuspidals.  I quote the statement verbatim from Kaletha's ICM talk \cite{Ka22}:

\begin{thm}\label{KalethaSth}  Assume that $G$ splits over a tame extension of $F$ and $p$ does not divide the order of $W(G)$. The set of isomorphism classes of {\bf regular} supercuspidal representations of $G(F)$ is in a natural bijection with the set of $G(F)$-conjugacy classes of pairs $(S,\theta)$, where $S$ is an elliptic maximal torus that splits over a tame extension, and $\theta$ is a regular character of $S(F)$.
\end{thm}

{\it Regularity}, as defined by Kaletha, is a property of general position.  The $L$-parameters that Kaletha assigns to regular supercuspidals are include those
homomorphisms $\varphi:  W_F \ra {}^LG$ such that (i) the image of wild inertia is contained in a maximal torus of $\hat{G}$ and (ii) the centralizer in $\hat{G}$ 
of the image of the inertia subgroup $I_F$ is abelian (the precise definition is slightly more general; see \cite[Definition 5.2.3]{Ka19}).  Kaletha's proposed formula for $\CL(\pi)$, when $\pi$ is associated to $(S,\theta)$ as in Theorem \ref{KalethaSth},
naturally begins with the parameter $\CL(\theta):  W_F \ra \hat{S}$ from the abelian Langlands correspondence, but the complete construction, which occupies 15 difficult
pages of \cite{Ka19}, is much too intricate to review here.  Moreover, Kaletha not only defines $\CL(\pi)$ but shows that the fiber of his
map $\CL$ has the structure predicted by Conjecture \ref{Vogan}.  
In \cite{Ka20} Kaletha singles out a somewhat more general class of Yu's supercuspidals that he calls {\it non-singular}, and carries out similar constructions for
their parameters.

Kaletha's construction raises two immediate questions:
\begin{question}
\begin{itemize}
\item[(1)]  Does Kaletha's parametrization coincide with those of Genestier-Lafforgue and Fargues-Scholze?
\item[(2)]  Do the $L$-packets he defines satisfy the endoscopic character identities of Property \ref{p9}?  In particular, do his regular supercuspidal $L$-packets correspond
to stable distributions?
\end{itemize}
\end{question}

Except for $GL(n)$, where the question has been answered affirmatively by combining the results of \cite{BH10} and \cite{MY}, very little is known about the first question.  The
second question, on the other hand, was recently settled by Fintzen, Kaletha, and Spice for $p$-adic fields, with $p >> 0$ \cite{FKS}.

\begin{remark}[Depth zero representations]  Kaletha's parametrization of regular supercuspidals generalizes the earlier study of depth zero supercuspidals by DeBacker-Reeder
and Kazhdan-Varshavsky \cite{DR09,KV}.  These earlier papers already established the stability of the depth zero supercuspidal $L$-packets under certain conditions.
\end{remark}

\section{Geometric constructions}\label{coho}


\subsection{The parametrizations of Genestier-Lafforgue and Fargues-Scholze}

The first definition of semisimple local parametrizations $\CL$ for general reductive groups was defined by Genestier and V. Lafforgue, in the 
setting of  Lafforgue's (semisimple) parametrization of cuspidal automorphic representations over function fields.    If $K$ is a global function field,
we let $\CA_0(G,K)$ denote the set of cuspidal automorphic representations of $G(\ad_K)$.

Here is the statement.

\begin{thm}\cite[Th\'eor\`eme 0.1]{GLa}\label{paramGL}
Let $F$ be a non-archimedean local field of characteristic $p$; $F \isoarrow k'(\!(t)\!)$ for some finite extension $k'$ of $\Fp$.  
Let $C$ be the coefficient field $\Qlb$, for any prime $\ell \neq p$.
There is a canonical parametrization with coefficients in $C$:
\begin{equation}\label{localparam}
\CL^{ss}:  \Pi(G/F)   \ra \Phi(G/F)^{ss}.
\end{equation}
The map satisfies Properties \ref{p1}-\ref{p7} and Property \ref{p7a}.  

It is also compatible
with the global semisimple parametrization 
\begin{equation}\label{VLparam}
\CL^{ss} = \CL^{ss}_{G,K}:  \CA_0(G,K) \ra \F(G/K)^{ss}
\end{equation}
of cuspidal automorphic representations
defined by Vincent Lafforgue \cite{Laf18} (Property \ref{p11}).

\end{thm}

Not all of the properties in the statement of Theorem \ref{paramGL} are stated explicitly in the article \cite{GLa}, but those
that are not can easily be deduced from the corresponding properties of the global parametrization and Chebotarev density.  If the stable (twisted) trace
formula were available over function fields, local endoscopic transfer (Properties \ref{p9} and \ref{p10}) should similarly follow from 
Chebotarev density by the known results at unramified places.  It seems likely that Property \ref{p12} can be proved for the local parameters, even though
it is not known in general for Lafforgue's global parameters; see \cite[\S 7.7]{GLo} for a proof in certain cases.

The article \cite{GLa}
contains a section constructing local and global theories of automorphic $L$-functions, but it does not satisfy point (d) of Definition \ref{autoL}:  
it is defined by means of the functorial transfer to $GL(n)$ via $\sigma$ that always exists over function fields, thanks to \cite{Laf02}.  Thus
it is not what we mean by a theory of automorphic $L$-functions.   On the other hand, \cite[\S 5.3]{GLo} proves a version of Property \ref{p8} for
the local $\gamma$-factors.

The construction in \cite{GLa} follows the pattern initiated in \cite{Laf18} in the global setting.  Instead of defining the parameter directly on some cohomology
group, Lafforgue in \cite{Laf18} attaches a pseudocharacter with values in ${}^LG$ to a given cuspidal automorphic $\Pi$ by using the cohomology of an infinite collection 
of moduli stacks of $G$-shtukas with legs, and applying results from geometric invariant theory to deduce that this information is equivalent to an actual semisimple parameter.
Genestier and Lafforgue use the same procedure with moduli spaces of local $G$-shtukas.  
Thus the construction qualifies as cohomological although the actual representation of the global Galois group on the cohomology is not determined.

More recently, Fargues and Scholze have defined local shtukas for $p$-adic fields in \cite{FS}.  From 
a categorical version of the local Langlands correspondence, in the spirit of the work of Gaitsgory and his collaborators,
they derive the following analogue of Theorem \ref{paramGL}:

\begin{thm}\cite{FS}\label{paramFS}  Let $F$ be a $p$-adic field.  Let $C$ be the coefficient field $\Qlb$, for any prime $\ell \neq p$.
There is a canonical parametrization with coefficients in $C$:
\begin{equation}\label{localparamFS}
\CL^{ss}:  \Pi(G/F)   \ra \Phi(G/F)^{ss}.
\end{equation}
The map satisfies Properties \ref{p1}-\ref{p7} and \ref{p7a} \cite[\S IX.6]{FS}.  Moreover, when $G = GL(n)$, the map $\CL^{ss}$
recovers the local Langlands correspondence of Theorem \ref{LLCGLN} \cite[\S IX.7.3]{FS}.
\end{thm}



\subsection{Close local fields}\label{CLF}

Let $F = k((t))$ be a local field of characteristic $p$ and $F^\sharp$ a $p$-adic field that is {\it $n$-close} to $F$ for some $n >> 0$:
$$\CO_F/m_F^n \isoarrow \CO_{F^\sharp}/m_{F^\sharp}^n,$$
where $m_F \subset \CO_F$ denotes the maximal ideal and $m_{F^\sharp}$ is defined analogously.

A Weil group parameter $\varphi \in \Phi(G/?)$, $? = F, F^\sharp$ is said to be of depth $n \in \NN$ if
$n$ is the maximum integer such that $\varphi$ is trivial on the subgroup $I_{F^?}^n$ of the inertia group $I_{F^?}$, where we are using the upper numbering.

Let $\Phi^n(G/?) \subset \Phi(G/?)$ denote the subset of (equivalence classes of) Weil group parameters of depth $n$.
\begin{thm}[Deligne]\label{delnn}   Assume $F$ and $F^\sharp$ are $n$-close as above.  Then there is a natural
bijection
$$\Phi^n(G/F) \isoarrow \Phi^n(G/F^\sharp).$$
\end{thm}

Deligne's theorem has a counterpart on the automorphic side.  
Suppose $G$ is split connected, and let $I_n(F) \subset G(\CO_F)$ (resp. $I_n(F^\sharp) \subset G(\CO_{F^\sharp})$) denote 
the  $n$-th Iwahori filtration subgroup, as defined in \cite[\S 3]{G}.  Let $H(G(?),n) = H(G(?),I_n(?))$ denote the depth $n$ Hecke 
algebras of $I_n(?)$-biinvariant functions on $G(?)$, with $? = F, F^\sharp$.   
We let $\Pi^n(G/?) \subset \Pi(G/?)$ denote the subset of (equivalence classes of)  irreducible smooth representations $\pi$ of $G(?)$
such that $\pi^{I_n(?)} \neq 0$.  Any $\pi \in \Pi^n(G,F^?)$ is then determined up to isomorphism by the representation of 
$H(G(F^?),n)$ on its invariant subspace $\pi^{I_n(F^?)}$.

The following is Ganapathy's refinement of a theorem of Kazhdan:

\begin{thm}[Ganapathy-Kazhdan, \cite{G}]\label{kazg}  
Let $n > 0$ and suppose $F^\sharp$ and $F$ are $n$-close.  Then 
there is a natural isomorphism
 $$H(G(F),n) \isoarrow H(G(F^\sharp),n)$$
and a natural bijection
$$\Pi^n(G/F) \isoarrow \Pi^n(G/F^\sharp)$$
that commutes with the actions of the depth $n$ Hecke algebras on the two sides.
\end{thm}

\begin{conj}[\cite{GHS}]\label{closel}   For any positive integer $n$, there is an $r \geq n$ such that the
following diagram commutes:
\begin{equation*}
\begin{CD}   \Pi^n(G/F)  @>Genestier-Lafforgue>>  \F^r(G/F) \\
 @V\text{ Thm \ref{kazg} } VV    @VV \text{ Thm \ref{delnn} }V \\
 \Pi^n(G/F^\sharp)  @>>Fargues-Scholze>  \F^r(G/F^\sharp)
 \end{CD}
\end{equation*}
\end{conj}

\begin{remark}\label{rn}  Since $\F(G/F^?) = \cup_r \F^r(G/F^?)$, there always exists some $r$ for which the
horizontal maps are defined in Conjecture \ref{closel}.   One can strengthen the above conjecture by requiring that $r$ be uniformly bounded in terms of $n$ and (split semisimple) $G$, or even that we can take $r = n$, as is known to be the case for $GL(n)$ \cite[Theorem 2.3.6.4]{Yu09}.
\end{remark}

\begin{question}  The Fargues-Scholze correspondence is not known to be independent of $\ell$.  Can this independence be proved without
proving Conjecture \ref{closel}?
\end{question}

\vfill
\pagebreak

\part{Incorrigible representations}

In the remainder of the paper we will assume $F = k((t))$ is a non-archimedean local field of positive characteristic $p$.  We will make use of the
local parametrization $\CL^{ss}$ of Theorem \ref{paramGL}, which is known to have properties   \ref{p1}-\ref{p7} and \ref{p7a}, and most importantly
is compatible with the global parametrization of Vincent Lafforgue (\ref{p11}).  The latter property allows us to to apply global methods to derive additional
information about $\CL^{ss}$.  Some of this information is unconditional and is recalled in \S \ref{cGHS}.  More can be said if we allow ourselves to use
the expected properties of the stable twisted trace formula that remain to be established over function fields, completing the results of \cite{LL}.  This will be the subject
of \S \ref{cIncorr} and Part \ref{Strategy}.

It has already been observed that Henniart's Splitting Theorem \ref{split} is the starting point for nearly all proofs of the local Langlands correspondence for $GL(n)$.
The next two sections are devoted to formulating a generalization of this theorem to all reductive groups.  We define a supercuspidal representation $\pi$ of the reductive group $G$  over $F$ to be {\it incorrigible} if it violates the conclusion of Henniart's Theorem \ref{split}.  The precise condition, which depends on realizing $\pi$ as a local component of a cuspidal automorphic representation,  is stated as  Definition \ref{incorrigible} below.  The formulation of this condition, like the statement of Theorem \ref{split}, presupposes
the existence of cyclic base change, which in turn presupposes the formalism of the stable trace formula.  The existence of such a formula has not yet been established for groups over function fields, although base change has been constructed by Henniart and Lemaire in the specific case of $GL(n)$ \cite{HeLe}.  The experts understand how to 
prove the necessary results for a general $G$ when the characteristic $p$ of $k$ is sufficiently large relative to $G$; it probably suffices to assume $p$ prime to
the order of the Weyl group of $G$.  In formulating our base change hypotheses in \S \ref{bchyp} we assume these problems have been resolved. 

\section{Existence of incorrigible representations}\label{cIncorr}

\subsection{Assumptions about base change}\label{bchyp}

\begin{struc}[Base change, version 1]\label{BC1}  For any extension $F'$ of $F$ and any cyclic extension $F''/F'$ with Galois group $\Gamma$, we let $\Gamma$ act on $\Pi(G/F'')$ by the Galois action on $G(F'')$.  Let $\mathcal{P}(\Pi(G/F'))$ denote the set of fibers of the map $\CL^{ss}: \Pi(G/F') \ra \F(G/F')$ (packets).  We have the following properties:

(i)  If $\varphi'' \in \F^{ss}(G/F'')$ is the restriction to $W_{F'}$ of some $\varphi' \in \F^{ss}(G/F')$, then $\Pi_{\varphi''}$ is a union of $\Gamma$-orbits in $\Pi(G/F'')$.  

(ii) There is a map
$$BC_{F''/F'}:   \mathcal{P}(\Pi(G/F')) \ra \mathcal{P}(\Pi(G/F''))^{\Gamma}$$
where $\mathcal{P}(\Pi(G,F''))^{\Gamma}$ are the packets that are invariant under the action of $\Gamma$.

(iii) Suppose $\pi \in \Pi(G/F')$ is a spherical representation, and let $[\pi] \in \mathcal{P}(\Pi(G/F'))$ denote the packet containing $\pi$.  Then $BC_{F''/F'}([\pi])$ contains the representation of $G(F'')$ obtained from $\pi$ by unramified base change.

(iii)  More generally, $BC_{F''/F'}$ is compatible with parabolic induction, in the sense of Property \ref{p7}.
\end{struc}

It seems likely that Structure \ref{BC1} will eventually be established by applying the full stable twisted trace formula.  However, it is possible that the following much weaker version can be proved using the techniques of \cite{Lab}, without full stabilization, once Arthur's simple trace formula has been established over function fields.
The next version  seems to suffice for our purposes.  

\begin{struc}[Base change, version 2]\label{BC}  For any extension $F'$ of $F$ and any cyclic extension $F''/F'$ with Galois group $\Gamma$, let $P(\Pi(G/F''))$ denote the set  of subsets of $\Pi(G/F'')$.  
Then 

(i) There is a map 
$$BC_{F''/F'}:  \Pi(G/F') \ra P(\Pi(G/F'')).$$

(ii) Suppose $\pi \in \Pi(G/F')$ is a spherical representation.  Then $BC_{F''/F'}(\pi)$ contains the representation of $G(F'')$ obtained from $\pi$ by unramified base change.

(iii)  More generally, $BC_{F''/F'}$ is compatible with parabolic induction, in the sense of (iii) of Property \ref{p7}.
\end{struc}

It is not assumed in (i) that the set $BC_{F''/F'}(\pi)$ is a {\it finite} subset of $\Pi(G/F'')$, though of course that is expected.  This is because one expects
to define cyclic base change by realizing $\pi$ (up to inertial equivalence)\footnote{This should always be possible for supercuspidal representations.  One then shows by a global argument that the base change is compatible with twist by characters, and this easily implies that the map $BC_{F''/F'}$ extends uniquely to a map of supercuspidal representations themselves, and not just of their inertial equivalence classes.  The map then extends uniquely to general representations by point (iii).} as a local component of a global automorphic representation at some place $v$. Structure \ref{BC} doesn't assume that the set of local components at $v$ of global base change obtained in this way are independent of the globalization.  Nevertheless, one can use this much weaker version of base change to define incorrigible representations.

For reasons that will be explained in the following section, the following property of the base change map has been separated from the other three.

\begin{conj}[Compatibility]\label{compat}  Suppose $F, F', F''$, and the map $BC_{F''/F'}$ are as in Structure \ref{BC}.  Then for any $\pi \in \Pi(G/F')$,
$$\CL^{ss}_{F''}(BC_{F''/F'}(\pi)) = \CL^{ss}_{F'}(\pi) ~|~_{W_{F''}}.$$
\end{conj}

\subsection{Incorrigible representations}\label{incorrigi}

We work with Structure \ref{BC}.  Although the domain of $BC_{F''/F'}$ is initially taken to be $\Pi(G/F')$, we extend it in the naive way to a map
$$P(\Pi(G/F')) \ra P(\Pi(G/F'')).$$ 
Thus for a subset $A \subset \Pi(G/F')$, $BC(A)$ is just the union of $BC(\pi)$ for all $\pi \in A$.    
If $F_0 \subset F_1 \subset \dots \subset F_r$ is a sequence of cyclic extensions we then define
$$BC_{F_r/F_0} = BC_{F_r/F_{r-1}} \circ \dots \circ BC_{F_1/F_{0}}$$
as maps of sets of representations.

\begin{definition}\label{incorrigible}  Let $\pi_0 \in \Pi(G/F_0)$ be a supercuspidal representation of $G(F_0)$, $\Pi_0 \subset \Pi(G/F_0)$ the packet containing $\pi_0$.  We say $\pi_0$ is {\bf incorrigible} if, for any sequence $F_0 \subset F_1 \subset \dots \subset F_r$ of cyclic extensions the base change packet $BC_{F_r/F_0}(\Pi_0)$ contains a supercuspidal member.
\end{definition}

We have defined pure representations $\pi \in \Pi(G/F)$ in Definition \ref{pure} (d).   Here is the conjectured generalization of Henniart's Splitting Theorem \ref{split}.

\begin{conjecture}\label{noincorr_conj}  There are no pure incorrigible supercuspidal representations.
\end{conjecture}

Examples are known of cuspidal unipotent representations whose semisimple parameters are unramified.  The purity hypothesis is not meant to apply to such representations.
I expect that cuspidal unipotent representations do not have pure semisimple parameters, but I do not know how this would be proved in general.
The following strengthening of Conjecture \ref{noincorr_conj} should be a consequence of a version of the local Langlands conjecture for $G$ that includes  compatibility with parabolic induction.  

\begin{conjecture}\label{Iwahori2}  Let $\pi_0 \in \Pi(G/F_0)$ be any  supercuspidal representation and let $\Pi_0 \subset \Pi(G/F_0)$ the packet containing $\pi_0$.  There is a finite sequence of cyclic extensions $F_0 \subset F_1 \subset \dots \subset F_r$ such that every member of $BC_{F_r/F_0}(\Pi_0)$ contains an Iwahori-fixed vector.
\end{conjecture}

\subsection{Incorrigible representations and $L$-functions}\label{lfunction}

In this section $\mathbf{G}$ is a connected reductive algebraic group over the function field $K = k(X)$.  We fix a point $v \in |X|$ and let $F = K_v$ denote the corresponding local field, $G = \mathbf{G}(F)$.  We assume  Structure \ref{BC} is available for $G_F$. 

\begin{thm}\cite{Laf18,GLa}\label{chtouca}  Conjecture \ref{compat} is valid for the group $G_F$ and the Genestier-Lafforgue parametrization $\CL^{ss} = \CL^{ss}_F$
of Theorem \ref{paramGL}.
\end{thm} 

\begin{proof}  Theorem \ref{paramGL} specifies that $\CL^{ss}$ satisfies Properties \ref{p3}, \ref{p7}, and \ref{p7a}.  It then follows easily, as in the footnote above,
that it suffices to verify compatibility for supercuspidal representations $\pi \in \Pi(G/K)$.   By Property \ref{p3} again, $\pi$ may be assumed to have central character of finite order.   Thus let $\pi_v \in \Pi(G/F)$ be supercuspidal with central character of finite order.  By \cite{GLo} we can choose $K$ as above with point $v \in |X|$,
and a global cuspidal automorphic representation $\pi$ of $\mathbf{G}(\ad)$, with the chosen supercuspidal $\pi_v$ as component at $v$.  Moreover, it is proved
in \cite[\S 7.3.3]{BFHKT} that we may assume that the global parameter attached to $\pi$ has Zariski dense image.  Let $F'/F$ be a cyclic extension; we assume $K'/K$ is a cyclic extension that is inert at $v$, with $K'_v = F'$.  Let $w \in |X|$ be a point that is unramified in $K'$, and let $w'$ be a point of $K'$ over $w$.  Compatibility as in Conjecture \ref{compat} for $BC_{K'_{w'}/K_w}$ follows from Property \ref{p7a} and Structure \ref{BC}.  Since the global parameter has Zariski dense image, compatibility for $BC_{F'/F}(\pi_v)$ then follows from Chebotarev's density theorem as in the proof of \cite[Proposition 6.4]{BHKT}.
\end{proof}


Let $\CA_0(\bG,K)_1$ denote the space of cuspidal automorphic forms on 
$$[\bG_K] := \bG(K)\backslash \bG(\ad_K)$$
 with unitary central character.
\begin{hyp}\label{llfcn}  Assume there is a faithful representation $\rho:  {}^LG \ra GL(V_\rho)$, an algebraic group $H$ over $K$, and, for every finite extension $K'/K$ and every place $v$ of $K'$ a space $S(K'_v)$ of locally constant $C$-valued functions $\Phi_v$ on $H(K'_v)$,  and a $G(\ad)$-invariant subspace  $\CA_1(\bG,K') \subset \CA_0(\bG,K')_1$, 
with the following properties.
\begin{itemize}  
\item[(a)]  For any finite extension $K'/K$, there is a family of integrals $Z(s,f,\Phi,\chi)$ that converges absolutely for $f \in \CA_1(\bG,K')$, 
$\Phi \in S(K') := \otimes'_v S(K'_v)$, $Re(s) > 1$, and $\chi$ a Hecke character of $GL(1)_{K'}$.   Here it is assumed there is a non-zero base point $\Phi_{0,v} \in S(K'_v)$ for almost all $v$ and the restricted direct product is taken with respect to these base points.
\item[(b)]  The zeta integrals are Eulerian:  if $f = \otimes_v f_v$ is a factorizable vector in an irreducible cuspidal automorphic representation $\pi \subset \CA_1(\bG,K')$ and $\Phi = \otimes_v \Phi_v$ is factorizable, then
$$Z(s,f,\Phi,\chi) = \prod_v Z(s,f_v,\Phi_v,\chi_v)$$
\item[(c)]  There is an Euler product $d(s,\chi,\xi_\pi) = \prod_v d(s,\chi_v,\xi_{\pi,v})$, depending only on the central character $\xi_\pi$ of $\pi$ such that, if $\Phi_v = \Phi_{0,v}$, $f_v$ is the spherical vector in $\pi_v$ for a hyperspecial maximal compact subgroup, and $\chi_v$ is unramified, then 
$$d(s,\chi_v,\xi_{\pi,v})\cdot Z(s,f_v,\Phi_v,\chi_v) = L(s,\pi_v,\chi_v,\rho)$$
is the Langlands Euler factor attached to $\pi_v$, $\chi_v$, and $\rho$.  
\item[(d)]  Given $\pi$ as above, for every place $v$ the family of zeta integrals $Z(s,f_v,\Phi_v,\chi_v)$ has a least common denominator $L(s,\pi_v,\chi_v)$ and satisfies a local functional equation, as in Tate's thesis.  More precisely, we suppose that any non-trivial additive character $\psi_v:  K_v \ra \CC^\times$ determines an involution $\Phi_v \mapsto \hat{\Phi}_v$ of $S(K_v)$.  Moreover, we suppose there are isomorphisms
$$c(\pi_v):  \pi_v \isoarrow \pi^\vee_v; f_v \mapsto f^{\vee}_v$$
satisfying $c(\pi^\vee_v)\circ c(\pi_v) = Id_{\pi_v}$ for any $\pi_v$.  Then 
\begin{equation}\label{funeq}  \frac{Z(s,f_v,\hat{\Phi}_v,\chi_v)}{L(s,\pi_v,\chi_v)} = \varepsilon(s,\pi_v,\chi_v,\psi_v)\cdot \frac{Z(1-s,f^{\vee}_v,\Phi_v,\chi_v^{-1})}{L(1-s,\pi^{\vee}_v,\chi_v^{-1})}.
\end{equation}
\item[(e)]  The global $L$-function $L(s,\pi,\chi) = \prod_v L(s,\pi_v,\chi_v)$ has a meromorphic continuation with at most finitely many poles, and satisfies a global functional equation
$$  L(1-s, \pi^\vee,\chi^{-1}) = \varepsilon(s,\pi,\chi)L(s,\pi,\chi),$$
where $\varepsilon(s,\pi,\chi) = \prod_v \varepsilon(s,\pi_v,\chi_v,\psi_v)$ whenever $\psi = \prod_v \psi_v$ is an additive character of $\ad_K/K$.

\item[(f)]  If $\pi_v$ is supercuspidal, then there is an integer $d < \dim \rho$  such that the set of characters $\chi_v$ for which the local Euler factor 
$L(s,\pi_v,\chi_v)$ has a pole is finite.  Moreover,
if $d(\chi_v)$ is the sum of the order sof the poles of $L(s,\pi_v,\chi_v)$ then $\sum_{\chi_v} d(\chi_v) \leq d$.
\item[(g)] Fix $v$ and define the local gamma-factor
$$\gamma(s,\pi_v,\chi_v,\psi_v) = \varepsilon(s,\pi_v,\chi_v,\psi_v)\frac{L(1-s,\pi^{\vee}_v,\chi_v^{-1})}{L(s,\pi_v,\chi_v)}.
$$
Given two irreducible admissible representations $\pi_{1,v}$ and $\pi_{2,v}$, there is an integer $N$ such that, for all $\chi_v$ of conductor at least $N$, we have
$$\gamma(s,\pi_{1,v},\chi_v,\psi_v) = \gamma(s,\pi_{2,v},\chi_v,\psi_v).$$
\end{itemize}
\end{hyp}

\begin{prop}\label{corrig}  Assume Hypothesis \ref{llfcn}.  Moreover, assume 
\begin{itemize}
\item[(*)]for any $\pi \in \CA_1(\bG,K') \subset  \CA_0(\bG,K')_1$, any finite place $v$ of $K$, with $F = K_v$, and any sequence $F = F_0 \subset F_1 \subset \dots \subset F_r$ with $F_i/F_{i-1}$ cyclic, there is a sequence $K' = K_0 \subset K_1  \subset \dots \subset K_r$ of cyclic Galois extensions, with a unique prime $v_i$ lying over $v$ and with $Gal(K_i/K_{i-1}) = Gal(F_i/F_{i-1})$, global base change $BC_{K_i/K_{i-1}}$ takes the base change of $\pi$ to a representation in $\CA_1(\bG,K_i)$; in particular, the image of $\pi$ under base change remains cuspidal.  
\end{itemize}
Let $\sigma$ be a pure supercuspidal representation of $G(K_v)$ that satisfies
\begin{itemize}
\item[(**)] $\sigma$ occurs, up to inertial equivalence, as a local factor of a cuspidal automorphic representation of
$G$ in the space $\CA_1(\bG,K')$.  
\end{itemize}
Then $\sigma$ is not incorrigible.
\end{prop}
\begin{proof}  Suppose $\sigma$ is a pure incorrigible supercuspidal representation of $G$ that occurs as a local factor in $\CA_1(\bG,K')$.  By a simple reduction, using Property \ref{p3} of Theorem \ref{paramGL}, we may assume $\sigma$ has central character of finite order.  As in the proof of Theorem \ref{chtouca}, we can use the results of  \cite{GLo} to globalize $\sigma$ to a cuspidal automorphic representation $\pi$ of $\bG$ with central character of finite order.   Let  $\varphi_\pi:  \Gamma_K \ra ^LG(C)$ denote its semisimple Langlands-Lafforgue parameter, $\varphi_\sigma:  \Gamma_F \ra ^LG(C)$ the local parameter.  Let $I_v \subset \Gamma_v$ be the
inertia group.  By the purity hypothesis the image of $I_v$ under $\varphi_\sigma$ is of finite order, and necessarily solvable.  We can thus find a sequence $K = K_0 \subset K_1 \subset \dots K_i \subset K_{i+1} \subset K_r$ of Galois extensions, with $K_{i+1}/K_i$ cyclic, such that for every prime $v'$ of $K_r$ dividing $v$, the image of the inertia group $I_{v'}$ under the restriction of $\phi_\pi$ to $\Gamma_{K_r}$ is trivial.  By assumption, and by Theorem \ref{chtouca}, there is an automorphic representation $\pi_r$ of $\bG_{K_r}$ which is supercuspidal at some prime $v'$ of $K_r$ dividing $v$, and whose Langlands-Lafforgue parameter satisfies
$$\phi_{\pi_r} = \phi_\pi ~|~_{\Gamma_{K_r}}.$$

Now let 
$S$ be the set of places of $K_r$ at which $\pi_r$ is ramified.  Because we are working over function fields, we know that the partial $L$-function $L^S(s,\pi_r,\chi,\rho) = L^S(s,\rho\circ\phi_{\pi_r},\chi)$ extends to a complete $L$-function $L(s,\rho\circ\phi_{\pi_r},\chi)$
that has a meromorphic continuation with functional equation.  On the other hand, by (d) and (e) of Hypothesis \ref{llfcn}, the complete $L$-function
$L(s,\pi_r,\chi,\rho)$ is also meromorphic with functional equation.
Moreover, by the stability hypothesis (g), the usual argument \footnote{For example, see Corollary 5.6 and the proof of Theorem 5.1 of \cite{GLo}.  Gan and Lomel\'i work with Langlands-Shahidi $\gamma$-factors, but the same argument works with  zeta integrals as in Hypothesis \ref{llfcn}.}  the local $\gamma$-factors are uniquely determined by the global functional equation.  By the purity hypothesis, the local Euler factors are also uniquely determined.  
In particular, as $\chi_v$ varies, letting $d(\chi_v)$ denote the order of the pole of $L(s,\pi_{v'},\chi_{1,v'},\rho) = L(s,\rho\circ\phi_{\pi_r}~|~_{\Gamma_{v'}}, \chi_{1,v'})$ at $s = 1$,
we have
$$\sum_{\chi_v} d(\chi_v) = \dim V_\rho.$$
But this contradicts (f) of Hypothesis \ref{llfcn}.
\end{proof}

\begin{remark}\label{double}  In the applications, $G$ will be of the form $G_0 \times G_0$, $\rho$ will be an irreducible representation of the dual group $\hat{G} = \hat{G}_0 \times \hat{G}_0$ that is trivial on the second factor, and  $\CA_1(\bG,K')$ will be the sum of representations $\pi_0 \otimes \pi_0^{\vee}$ where $\pi$ runs over cuspidal representations of $G_0$ with unitary central character.  Then the involutions $c(\pi_v) = c(\pi_{0,v}\otimes \pi_{0,v}^\vee)$ will just switch the two factors.  The hypothesis (*) of Proposition \ref{corrig} can then be guaranteed as in the proof of Proposition \ref{compat-prop} below.  Moreover, every supercuspidal representation of the form 
$\sigma = \pi_v \otimes \pi_v^{\vee}$ satisfies (**), which suffices for the applications.
\end{remark}

When $d = 0$ in (f), Proposition \ref{corrig} can be strengthened:

\begin{prop}\label{corrig2}  Assume Hypothesis \ref{llfcn} and the condition (*) in Proposition \ref{corrig}.  Moreover, assume the integer $d = 0$ in point (f).  Then no supercuspidal representation of $G$ is incorrigible.
\end{prop}

\begin{proof}  As in the proof of Proposition \ref{corrig}, we may reduce to the case that the image of the inertia group $I_v \subset \Gamma_v$ under the semisimplification of  $\rho\circ \phi_{\pi_r}$ is trivial, and in particular we may assume $r = 0$ and $\pi_r = \pi$.  It then suffices to show that the local factor at $v$ of $L(s,\rho \circ \phi_{\pi},\chi)$ has at least one pole for some $\chi$.   But it is well-known that the local $L$-factor of any Weil-Deligne parameter for $GL(\dim V_r)$ that is trivial on the Weil group has at least one pole.
\end{proof}

\subsection{The case of $GL(n)$}\label{gln-section}

Conjecture \ref{Iwahori2} is a well known consequence of the local Langlands correspondence for $GL(n)$ over any non-archimedean local field.  Over local fields of positive characteristic it is one of the many consequences of Laurent Lafforgue's global Langlands correspondence \cite{Laf02}, and the considerations of \S \ref{lfunction} are entirely superfluous.  We nevertheless show that the hypotheses of Proposition \ref{corrig} apply in this case, using only the methods of \cite{Laf18} that are available for all groups.

\begin{prop}\label{propgln}  Let $D$ be a simple algebra over the global function field $K$ of dimension $m^2$, $G = GL(a,D) \times GL(a,D)$ for some integer $a \geq 1$, and $F = K_v$ for some $v$.  Then the Godement-Jacquet zeta integral of \cite{GJ72} satisfies the conditions of Hypothesis \ref{llfcn}.  In particular, if $G$ is the multiplicative group of a simple algebra over $F$ then $G$ has no incorrigible supercuspidal representations.
\end{prop}
\begin{proof}  The Langlands dual group of $G$ is $GL(am,\CC) \times GL(am,\CC)$ and we take $\rho$ to be the standard representation of the first factor.
The conditions of Hypothesis \ref{llfcn} were directly modeled on the constructions and results of \cite{GJ72}: $H$ is the simple algebra $M(a,D)$, viewed as an additive algebraic group, $\Phi_v \mapsto \hat{\Phi}_v$ the Fourier transform relative to the character $\psi_v \circ Tr_{D_v}$, where $Tr_{D_v}$ is the reduced trace map.  Moreover, the subspace $\CA_1(\bG,K')$ and the involutions $c(\pi_v)$ are as in Remark \ref{double}.  In particular, the functions $f_v$ in the factorization are matrix coefficients of the representation $\pi_{0,v}$ in the space of locally constant functions on $GL(a,D_v)$ and $f_v^\vee(g) = f_v(g^{-1})$.

With these conventions, conditions (a)-(f) of Hypothesis \ref{llfcn} are proved in \cite{GJ72}, and the stability condition (g) is proved in \cite{JS}; see also \cite{Lo16} and \cite{HeLo13}.  Specifically, if $\pi_v$ is supercuspidal then the Godement-Jacquet Euler factor $L(s,\pi_v) = 1$, so condition (f) is automatic with $d = 0$.  Thus we can apply Proposition \ref{corrig2}, and Remark \ref{double}, to conclude.
\end{proof}

\subsubsection{Splitting for $GL(n)$ and the local Langlands correspondence}  
Proposition \ref{propgln} gives an apparently independent proof of Henniart's Splitting Theorem for $GL(n)$.  Since this turns out to be the crucial step
in the proof of the local Langlands correspondence, together with the existence of a canonical parametrization, Rapoport asked me in 2018 whether the argument by way of
$L$-functions represents a new proof of the local Langlands conjecture for $GL(n)$.  The answer is no, but for an interesting reason, which I describe briefly.

For function fields, the existence of the canonical parametrization, together with the known theory of $L$-functions for $GL(n)\times GL(m)$, 
was already used by Laurent Lafforgue in \cite{Laf02} to prove the entire global Langlands correspondence over function fields.  The article \cite{Laf02}
cited Henniart's numerical correspondence, but upon closer examination this turns out to be unnecessary.  Using the properties of Vincent Lafforgue's global parametrization
in order to apply the $L$-function criterion of Proposition \ref{llfcn}, and then using the latter to deduce the local Langlands correspondence, looks dangerously circular.

Suppose now that $F$ is a $p$-adic field that arises as a completion $F = K_v$ of a global number field $K$.  Shimura varieties attached to unitary similitude groups $\bG_n$ are used in \cite{HT01} and \cite{Sch13} to define local parametrizations of supercuspidal representations of $GL(n,F)$ by $n$-dimensional representations of $W_F$ that are compatible with global parametrizations of certain cuspidal automorphic representations of the Galois group of $K$, for certain $K$ (in particular, $K$ is a totally imaginary extension of a totally real field).  Let $\CL$ denote any of these (semisimple) parametrizations.  In \cite{H98} a class $\CA_2(\bG_N,K)$ of cuspidal automorphic representations is constructed, for any integer $N$, with the property that, for any $\pi \in \CA_2(\bG_N,K)$, the restriction of $\CL(\pi)$ to the Galois group of some extension $K'/K$ of degree $n$ is abelian and is attached to a Hecke character
of type $A_0$.  It is moreover proved that, up to twisting by an unramified character, every irreducible representation of $W_F$ occurs as a summand of the restriction to the 
decomposition group $D_v = Gal(\bar{F}/F)$ of $\CL(\pi)$, for some $\pi \in \CA_2(\bG_N,K)$ and some $N$; the proof follows Deligne's construction of local constants for the functional equations of Artin $L$-functions in \cite{De73}.  Henniart's proof in \cite{He00} works with the class of representations $\CA_2(\bG_N,K)$.  

We may thus define a subset $\Pi_2(GL(n)/F) \subset \Pi(GL(n)/F)$ to be the set of supercuspidal representations $\sigma$ such that, for some $N = n_0 + n_1 + \dots + n_r$,
with $n = n_0$, and some $\pi \in \CA_2(\bG_N,K)$, with $F \isoarrow K_v$, the local component $\pi_v$ of $\pi$ is the normalized induction of a representation
$\sigma_0 \otimes \sigma_1 \dots \otimes \sigma_r$ of the parabolic subgroup of $GL(N,F)$ corresponding to the partition $(n_0,n_1,\dots, n_r)$ of $N$,
with $\sigma_i$ a supercuspidal representation of $GL(n_i,F)$, and $\sigma$ and $\sigma_0$ are inertially equivalent.   The arguments of \cite{HT01} or \cite{Sch13}, together with \cite{H98} and \cite{BHK}, then show that
\smallskip

\begin{prop}\label{noincorrpi2}
~~

\begin{itemize}
\item  No $\sigma \in \Pi_2(GL(n)/F)$ is incorrigible;
\item  The local parametrization from $\Pi_2(GL(n)/F)$ to $n$-dimensional representations of $W_F$ is surjective onto the set of irreducible $n$-dimensional
representations of $W_F$.
\end{itemize}
\end{prop}

Henniart's proof in \cite{He00} is similar, except that he doesn't define a parametrization of all the supercuspidal representations of $GL(n,F)$.
Bijectivity of the local parametrization then follows in \cite{HT01} and \cite{He00} from Henniart's numerical correspondence, and in \cite{Sch13} from Scholze's
analysis relating the action of inertia on nearby cycles to invariants under the Iwahori subgroup.  If neither of these is assumed, then the arguments used to prove
Proposition \ref{noincorrpi2} -- especially Proposition \ref{propgln} -- suffice to prove surjectivity but not injectivity.  

The problem is that the arguments sketched above do not suffice to show that
$\Pi_2(GL(n)/F)$ contains all supercuspidal representations of $GL(n,F)$, and thus provide no information about the local $L$- and $\varepsilon$-factors
for supercuspidal representations not in $\Pi_2(GL(n)/F)$.
The difference between the proofs for number fields and function fields is that in the latter case every global parameter is known to be automorphic; thus its $L$-function has two independent definitions, one from arithmetic and one from the theory of automorphic forms, and the stability of $\gamma$-factors, together with compatibility at unramified places, implies that all local functional equations must
necessarily coincide.  For number fields this is only known for the parameters attached to $\pi \in \CA_2(\bG_N,K)$.

\subsection{Example:  classical groups}\label{classical}

Hypotheses are as in Section \ref{lfunction}:  $G$ is a connected reductive group over the function field $K = k(X)$.  More precisely, to adapt to the formalism of Hypothesis \ref{llfcn} we assume $G = G_0 \times G_0$, where $G_0$ is a classical group over $K$ of the form $Sp(2n)$ or $SO(V)$ for a vector space $V$ over $K$ with a non-degenerate symmetric bilinear form $b_V$.    We let $F = K_v$ for some place $v$.

\begin{prop}\label{propclassical}  Assume Hypotheses \ref{simpleBC} and \ref{labesse} hold for $G_0$ over $K$.  Then  the Piatetski-Shapiro-Rallis doubling integral of \cite{psr} satisfies the conditions of Hypothesis \ref{llfcn}.  In particular, if $G_0$ is a classical symplectic or special orthogonal group over $F$ then $G$ has no incorrigible pure supercuspidal representations.
\end{prop}
\begin{proof}  In the framework of Hypothesis \ref{llfcn}, we have $H = Sp(4n)$ or $H = SO((V\oplus V, b_V\oplus (-b_V))$, the involutions $c(\pi_v)$ are as in Remark \ref{double}, and the rest of the construction is as described in \cite{psr}.
The local theory of the doubling integral is worked out in  detail by Lapid and Rallis in \cite{LR}, and completed by Yamana in \cite{Ya14}.  Proposition 5 of \cite{LR} implies that, when $\pi$ is supercuspidal and $\omega$ is a unitary character, then the local factor $L(s,\pi, \omega)$ is holomorphic for $Re(s) \geq \frac{1}{2}$.  In particular, there is no cancellation between the poles of the numerator and denominator of the local gamma-factor $\gamma(s,\pi,\omega,\psi)$.  Now Theorem 4 of \cite{LR}, together with the stability property of Rallis and Soudry (\cite{RS}; see also \cite{Br} for the analogous case of unitary groups), implies the conditions of Hypothesis \ref{llfcn}.  
\end{proof}

\begin{remark}  The possible poles of local Euler factors of supercuspidal representations of unitary groups are determined in Theorem 6.2 of \cite{HKS}.  They correspond as expected to the classification of cuspidal unipotent representations.   In particular, the local Euler factor of a pure supercuspidal representation has at most a single simple pole.  The  case of orthogonal and symplectic groups is treated by Yamana in \cite{Ya14} (see the proof of Theorem 5.2 on p. 684).  
\end{remark}

Of course, if we are assuming Hypothesis \ref{simpleBC} then we could also assume the full stable twisted trace formula, as in Arthur's book \cite{A}, and then Proposition \ref{propclassical} follows from the case of $GL(n)$ by Arthur's trace formula arguments.   However, base change, in the form of Hypothesis \ref{labesse}, is considerably simpler to manage.

\subsection{Example:  $G_2$}\label{G2L}

Among exceptional groups, $G_2$ has the unique distinction of being the only one for which a theory of $L$-functions has been constructed that applies to all cuspidal automorphic representations.  In fact, there are two known constructions.  One, due to Ginzburg and Hundley \cite{GH}, is an analogue of the doubling method for classical groups, based on an embedding of $G_2 \times G_2$ in $E_8$, which supplies the degenerate Eisenstein series needed to prove the global meromorphic continuation.  The other construction is due to Gurevich and Segal \cite{GS1, GS2, Se} and is based on a degenerate Eisenstein series on $Spin(8)$.  

Both constructions include a verification that the unramified Euler factors are those attached to the standard $7$-dimensional representation of (the dual group) $G_2$.  However, they stop short of defining local $L$-factors at ramified places that satisfy a local functional equation as in (d) of Hypothesis \ref{llfcn}.  In particular, there is no a priori bound on poles of local $L$-factors of supercuspidal representations, as in (f) of Hypothesis \ref{llfcn}.  

Thus the strategy outlined in this note for eliminating incorrigible representations cannot yet be carried out for $G_2$.  We note, however, that the analogous global question has been addressed in \cite{GS2}, where conditions are found for the existence of poles of global $L$-functions of cuspidal automorphic representations of $G_2$.  

\subsection{Simple base change}\label{BCsection}  A version of Structure \ref{BC} has been constructed -- over number fields -- by Labesse in \cite{Lab}.  The construction is global and is worked out only under simplifying assumptions that eliminate parabolic and endoscopic terms in the stable (twisted) trace formula.  This is sufficient for local applications.  Labesse cannot make use of the parametrization, of course, so the base change map is defined purely by a character relation; however, when combined with Structure \ref{param} it can be reinterpreted.


Let $G$ be a connected reductive group over the global field $K$, which need not be a function field.  We assume the following version of Labesse's base change holds for $G/K$.
\begin{hyp}\label{simpleBC}  (i)  Let $K'/K$ be a finite cyclic extension.  Let $\CG_0(G,K)$, resp. $\CG_0(G,K')$, denote the set of cuspidal automorphic representations of $G_K$, resp. $G_{K'}$.  There is a (non-empty) subset $\CG_{simple}(G,K) \subset \CG_0(G,K)$ with the following property.  For any $\Pi \in \CG_{simple}(G,K)$, there is a non-empty set $BC(\Pi) \subset \CG_0(G,K')$ with the following property:  for every place $v$ of $K$ at which both $\Pi$ and $K'$ are unramified, and for any place $v'$ of $K'$ dividing $v$, we have
$\Pi'_{v'} \isoarrow BC(\Pi_v)$ for any $\Pi' \in BC(\Pi)$, where $BC(\Pi_v)$ is the base change of $\Pi_v$ under the local Langlands correspondence for unramified representations.

(ii) Moreover, for any place $v$ of $K$, with $v'$ as above, the set $BC(\Pi_v)$ consisting of local components  $\Pi'_{v'}$ at $v'$ of members $\Pi' \in BC(\Pi)$ depends only on $\Pi_v$ and not on the global representation $\Pi$.  

(iii) Finally, suppose $P \subset G$ is a rational parabolic subgroup with Levi factor $M$.  Suppose the global base change maps are defined for appropriate subsets $\CG_{simple}(M,K) \subset \CG_0(M,K)$ and $\CG_{simple}(G,K) \subset \CG_0(G,K)$.  Then  the local maps $\Pi_v \mapsto BC(\Pi_v)$ are compatible with parabolic induction from $M$ to $G$.
\end{hyp}

The subscript $_{simple}$ refers to the use of Arthur's simple  trace formula (STF) to establish global base change.  In \cite{Lab}, $K$ is a number field and the subset $\CG_{simple}(G,K)$ consists of representations that are Steinberg at some number of places (at least $2$).  This provides a sufficient condition to permit the application of the STF -- in other words to eliminate the difficult parabolic terms from both sides of the invariant trace formula.  The Steinberg condition also eliminates the endoscopic terms, thus substantially simplifying the comparison of trace formulas that implies Labesse's result.  We assume 
\begin{hyp}\label{labesse} Labesse's method works over function fields as well and we assume $\CG_{simple}(G,K)$ always includes at least the representations which have
Steinberg components at two places.
\end{hyp}

We derive a simple consequence of Hypothesis \ref{simpleBC}.  First, we show 
\begin{prop}\label{simplestructure} Hypothesis \ref{simpleBC} implies Structure \ref{BC}.  
\end{prop}
\begin{proof}   Let $F, F', F''$ be as in the statement of Structure \ref{BC}.  First let $\pi \in \CG(G,F')$ be a supercuspidal representation with central character of finite order.   Let $K'$ be a  global function field with a place $v$ such that $K'_v \simeq F'$.  Using the globalization method of \cite{GLo}  we see that there exists an automorphic representation $\Pi \in \CG_0(G,K')$ and two places $w, w'$ of $K'$ with $\Pi_v \isoarrow \pi$ and $\Pi_w$, $\Pi_{w'}$ both isomorphic to  Steinberg representations.  By Hypothesis \ref{labesse}, such a $\Pi$ belongs to $\CG_{simple}(G,K')$.   Let $\CG_{simple}(G,K')(\pi)$ be the set of all such $\Pi$.  Thus if $K''/K$ is a finite cyclic extension, there exists $BC_{K''/K'}(\Pi) \subset \CG_0(G,K'')$ as in Hypothesis \ref{simpleBC}.  For any $K''/K'$ as above in which a single prime $w$ divides $v$, and such that $K''_w \isoarrow F''$, we let be the set of representations $\Pi''_w$ for all $\Pi \in \CG_{simple}(G,K')(\pi)$ and all $\Pi'' \in BC_{K''/K'}(\Pi)$.  We define $BC_{F''/F'}(\pi) \subset \CG(G,F'')$ to be the union of the $BC_{K''/K'}(\pi)$ for all $K''/K'$ as above.   

This defines $BC_{F''/F'}(\pi)$ for supercuspidal representations with central characters of finite order.  If $\pi$ is any supercuspidal representation of $G(F')$, choose a character $\chi:  G(F') \ra C^\times$ such that $\pi\otimes \chi$ has character of finite order.  Then we let 
$$BC_{F''/F'}(\pi) = BC_{F''/F'}(\pi\otimes \chi)\otimes BC_{F''/F'}(\chi^{-1}),$$
where $BC_{F''/F'}(\chi^{-1})$ can be defined using the norm map.  It is easy to see by a global argument that this definition doesn't depend on the choice of $\chi$.  Thus we have defined $BC_{F''/F'}(\pi)$ for all supercuspidal representations.  If $M \subset P \subset G$ is as in (iii) of Hypothesis \ref{simpleBC},  if $\sigma$ is a supercuspidal representation of $M(F')$, and if $\pi$ is an irreducible constituent of the normalized induction $Ind_{P(F')}^{G(F')}\sigma$, then we let $BC(\pi)$ denote the set of constituents of
$Ind_{P(F'')}^{G(F'')}\sigma''$ as $\sigma''$ runs over elements of $BC_{F''/F'}(\sigma)$.   Properties (ii) and (iii) of Structure \ref{BC} follow immediately. 

\end{proof}

\begin{prop}\label{compat-prop}  Assume Hypotheses \ref{simpleBC} and \ref{labesse} and Structure \ref{param}.   

(a) Conjecture \ref{compat} is valid for $BC_{F''/F'}$ defined as in Proposition \ref{simplestructure}

(b) Moreover, suppose $F$ is of positive characteristic.  If $\pi_1, \pi_2 \in \CG(G,F')$ are supercuspidal representations such that
$$\CL^{ss}_{F'}(\pi_1) = \CL^{ss}_{F'}(\pi_2)$$
then, for any cyclic extension $F''/F'$, and for any $\pi'_1 \in BC(\pi_1)$, $\pi'_2 \in BC(\pi_2)$, we have
$$\CL^{ss}_{F''}(\pi'_1) = \CL^{ss}_{F''}(\pi'_2).$$
\end{prop}

\begin{proof}  By construction, Conjecture \ref{compat} is proved just as in Theorem \ref{chtouca}.   For (b), we argue analogously.  By applying Property (iv) of Structure \ref{param} we may reduce to the case where  $\pi_1$ has central character of finite order; then $\pi_2$ has the same central character by Property (v) of Structure \ref{param}.  Thus we may apply  globalization as in \cite{GLo} to find a  global field $K$ with places $v_1, v_2$ such that $K_{v_1} \isoarrow K_{v_2} \isoarrow F'$, and a cuspidal automorphic representation $\Pi$ of $G_K$ in $\CG_{simple}(G,K)$ such that $\Pi_{v_1} \isoarrow \pi_1$, $\Pi_{v_2} \isoarrow \pi_2$.  It follows that the global Langlands-Lafforgue parameter $\sigma$ of $\Pi$ has the property that the semisimple local parameters are equivalent:
$$\sigma ~|~_{W_{K_{v_1}}} \isoarrow \sigma ~|~_{W_{K_{v_2}}}$$
Now we find a cyclic extension $K'/K$ such that $K'_{v_1} \isoarrow K'_{v_2} \isoarrow F''$ and such that $BC(\Pi)$ exists and remains cuspidal.  We can do that, for example, by choosing auxiliary places $v_i$, $i = 3, 4, 5$ such that $\Pi_{v_3}$ is supercuspidal and choosing $K'$ to be split completely over $v_3$, and such that $\Pi_{v_4}$ and $\Pi_{v_5}$ are Steinberg, so that $\Pi \in \CG_{simple}(G,K)$.  Then (b) follows by Chebotarev density and a globalization argument as in \cite{BHKT}, as before.
\end{proof}

\section{Results of \cite{GHS}}\label{cGHS}

Let $K$ be a global function field over $k$ and let $\Pi \in \CA_0(G,K)$.  
Let $\sigma:  \hat{G} \ra GL(N)$ be any representation, $S$ the set of primes where $\Pi$ is ramified.  Then
$$\CL^{ss}(\Pi)_\sigma := \sigma\circ \CL^{ss}(\Pi):  Gal(K^{sep}/K)  \ra GL(N,\Qlb)$$
corresponds to a semi-simple $\ell$-adic local system $L(\Pi)_\sigma$ on $Y \setminus |S|$.   
\medskip

By Deligne's Weil II \cite{De80}, each irreducible summand of $L(\Pi)_\sigma$ is punctually pure, up to twist by a character of $\G_k := Gal(\bar{k}/k)$.      

It follows that, for any $w$, the eigenvalues of $\sigma\circ\CL_w(\Pi_w)(Frob_w)$ are Weil $q$-numbers of various weights (up to the twist, which we ignore).   
\medskip


If $G \neq GL(n)$, not all supercuspidals are pure, in the sense of Definition \ref{pure} (d).  Here is the main theorem of \cite{GHS}.

\begin{thm}\cite{GHS}\label{pureram}  Let $\pi$ be a {\bf pure} supercuspidal representation of $G(F)$, where $G$ is connected reductive and $G(F)$ is unramified.  Suppose $\pi$ is compactly induced from a compact open subgroup $U \subset G(F)$.    
Suppose moreover that $q > 3$.  Then the $\ell$-adic Genestier-Lafforgue parameter $\CL^{ss}(\pi)$ restricts to a non-trivial map from the inertia group $I_F$ to $\hG(\Qlb)$.  
\end{thm}   

\begin{cor}\label{prin}  Let $\pi$ be a pure irreducible (not necessarily supercuspidal) representation of $G(F)$, with $G$ split semisimple.  Assume $p$ does not divide
the order of the Weyl group of $G$.  Suppose $\CL^{ss}(\pi)$ is unramified.  Then $\pi$ is an irreducible constituent of an unramified principal series representation.
\end{cor}

\begin{proof}  Under our hypothesis on $p$, Fintzen's Theorem \ref{fintzenthm} implies that any supercuspidal representation of $G$ or any of its Levi subgroups is
compactly induced from a compact open subgroup of $G(F)$.  Since $\CL^{ss}$ is compatible with parabolic induction, the Corollary follows from Theorem \ref{pureram} and induction on the dimension of $G$.  
\end{proof}

Corollary \ref{prin} gives a positive answer to Question \ref{packets} when $\varphi$ is a pure unramified parameter -- the $L$-packet attached to $\varphi$,
meaning the set of $\pi$ with  parameter $\varphi$, is finite, provided $p$ is prime to the order of $W(G)$.   This is much weaker than what is really expected, or what can be proved
for $GL(n)$ or classical groups, but as far as I know there is no  general finiteness statement for ramified parameters.  Part \ref{Strategy} below is predicated on the
assumption that this very weak assertion may be almost enough to prove finiteness in general.

There is a variant of Theorem \ref{pureram}.

\begin{thm}\cite{GHS}\label{purewild}  Under the hypotheses of Theorem \ref{pureram}, suppose the compact open subgroup $U$ is sufficiently small; for example,
suppose it is contained in the principal congruence subgroup of a maximal compact subgroup.  Then $\CL^{ss}(\pi)$ is wildly ramified.
\end{thm}

\subsection{Pure Weil-Deligne parameters}\label{pureWD}

 Supercuspidal representations $\pi$ that are not pure, and discrete series representations more generally,
 are expected to  have Weil-Deligne parameters
$\varphi := (\rho,N)$ with $N \neq 0$.     In particular, the conclusion of Corollary \ref{prin} above can fail for such representations:  for groups other than $GL(n)$ and its inner forms, there are well-known constructions of {\it unipotent supercuspidals} with unramified semi-simple parameters.  

However, we expect that the pair $(\rho,N)$ 
satisfies {\bf purity of the monodromy weight filtration} ( up to unramified twist):  in other words, that the Weil-Deligne representaiton $(\rho,N)$ is {\bf pure} in the sense of Definition \ref{pure} (a).

To a Weil-Deligne parameter $(\rho,N)$ a standard construction  associates a continuous homomorphism $\varphi_{\rho,N}:  W_F \times SL(2,\CC) \ra {}^LG(\CC)$
that is algebraic on the second factor.  
\begin{defn}\label{tempWD}  The Weil-Deligne parameter $(\rho,N)$ is {\rm tempered} if the restriction of $\varphi_{\rho,N}$ to $W_F \times SU(2)$ has bounded image after projection to $\hat{G}(\CC)$.
The parameter $(\rho,N)$ is {\rm essentially tempered} if its image in $\hat{G}^{ad}$ is bounded.  
\end{defn}

The following is easy to check:
\begin{lemma}  The parameter $(\rho,N)$ is tempered if and only if it satisfies purity of the monodromy weight filtration.
\end{lemma}

\begin{thm}\cite{GHS}\label{MWF}  Assume $p > 3$ and let $\pi$ be any discrete series representation of $G(F)$.  Then its semisimple parameter $\CL^{ss}(\pi)$ has
a (necessarily unique) completion to an essentially tempered   Weil-Deligne parameter.
\end{thm}

The proof sketched below, based on Poincar\'e series, applies when $\pi$ is supercuspidal.  An argument due to Beuzart-Plessis, using the Deligne-Kazhdan
simple trace formula, allows us to prove Theorem \ref{MWF} for general discrete series parameters.  Combining Theorems \ref{MWF} and \ref{pureram}, we then have:

\begin{cor}\label{unram}  Assume $p > 3$.  Suppose $\pi \in \Pi(G/F)$ belongs to the discrete series.  Then either the unipotent factor of $\CL(\pi)$ is non trivial or
$\CL^{ss}(\pi)$ does not factor through  $W_F/I_F$ (or both).  

In particular, suppose $\CL^{ss}(\pi)$ is pure as in Definition \ref{pure} (b).  Then $\CL^{ss}$ does not factor through $W_F/I_F$.
\end{cor}

\subsection{About the proofs}

Most of the results of \cite{GHS} are based on the use of Poincar\'e series, as in \cite{GLo}, to realize a supercuspidal representation $\pi$ as the local component of a cuspidal
automorphic representation of $G$ over a global function field $K = k(Y)$, for some smooth projective curve $Y$.  The construction in \cite{GLo} used a matrix coefficient of $\pi$ as the local component $\varphi_v$ at some place $v$ of a function $\varphi = \prod_w \varphi_w$ on $G(\ad_K)$, where $w$ runs over places of $K$.  The Poincar\'e series is defined as
$$P_\varphi(g) = \sum_{\gamma \in G(K)} \varphi(\gamma \cdot g).$$
When $\pi$ is compactly induced from a compact open subgroup $U$, $\varphi_v$ can be taken to be supported in $U$.  When this is the case we can take the curve $Y$ to be
the projective line $\mathbb{P}^1$ over the finite field $k$.  Choosing data carefully at the $k$-rational points of $\mathbb{P}^1$, and assuming $\varphi_w$ to be the
characteristic function of a hyperspecial maximal compact subgroup when $w$ is not $k$-rational, we are able to deduce Theorem  \ref{pureram} from the classification of tame
local systems on $\Gm$ that follows from Deligne's Weil II.

When $\pi$ is supercuspidal, Theorem \ref{MWF} is proved by a similar globalization argument, but $Y$ can be arbitrary.  We construct a Poincar\'e series with any matrix
coefficient of $\pi$ at the point $x$ and the matrix coefficient of an appropriate supercuspidal representation at a second point $y$, whose local parameter is known to be irreducible.  We  use the Kloosterman representations of \cite{HNY}, which are known when $p > 3$ to have irreducible parameters; a different argument was used in \cite{GLo} when $G$
is a classical group.  Let $\Pi$ be the globalization thus obtained, with global parameter $\CL^{ss}(\Pi)$.  The irreducibility of the local parameter at $y$ implies that the composition of 
$\CL^{ss}(\Pi)$ with an appropriate representation of the $L$-group is irreducible.  Then Deligne's theorem on the purity of the monodromy weight filtration
applies at $x$ and defines the required completion to a Weil-Deligne parameter.

\vfill
\pagebreak

\part{A strategy to construct a local correspondence}\label{Strategy}

We sketch a construction based on the Genestier-Lafforgue parametrization and the ideas of \cite{BHKT}.  As mentioned in \ref{classicalsection}, the model is Arthur's construction \cite{A}
of the  local correspondence for classical groups by determining the fibers of twisted transfer to $GL(N)$.   Arthur's method consists in comparing the stable trace
formula for a classical group $G$ with a twisted trace formula for $GL(N)$.  The method relies heavily on multiplicity one for $GL(N)$, which guarantees that
the comparison of traces is non-trivial.  The method sketched here is an induction on cyclic base change.

\section{Hypotheses}

We let $G$ be a connected reductive group.  In order to avoid unnecessary complication we assume $G$ to be split but this should ultimately
be unnecessary.  We fix a non-archimedean local field $F = k((t))$ of characteristic $p > 0$, with $q = |k|$.   In what follows, $\ell$ is a prime number distinct from $p$
that will be assumed as large as necessary.   

Let $K$ be a function field.  For any automorphic representation $\Pi$ of $G(\ad_K)$, and any prime $\ell \neq p$,
we let $\CLs(\Pi)_\ell$ denote the global $\ell$-adic semisimple parameter attached by V. Lafforgue, and more generally by C. Xue, to $\Pi$:
\begin{equation}\label{CLs} \CLs(\Pi)_\ell :  Gal(K^{sep}/K) \ra {}^LG(\Qlb).
\end{equation}
If $\pi$ is an irreducible representation of $G(F)$ we similarly denote by $\CLs(\pi)$ the semisimple parameter constructed in \cite{GLa}
$$  \CLs(\pi):  W_F \ra {}^LG(C)$$
with values in any algebraically closed field $C$.

\begin{hyp}\label{multone}  Let $\Pi$ be an everywhere unramified cuspidal automorphic representation of $G(\ad_K)$ for some function field $K$.  Suppose its global $\ell$-adic
parameter $Gal(K^{sep}/K) \rightarrow \hG(\Qlb)$ has image equal to $\hG(\CO)$ for some $\ell$-adic integer ring $\CO$.  Then $\Pi$ has
global multiplicity one.  
\end{hyp}

\begin{hyp}\label{cyclic}  Existence of stable cyclic base change of essentially tempered automorphic representations.
\end{hyp}

Assuming Hypothesis \ref{cyclic}, it makes sense to make the following hypothesis:

\begin{hyp}\label{incorr}  Conjecture \ref{noincorr_conj} is valid:  there are no pure incorrigible representations, in the sense of \S \ref{cIncorr}.
\end{hyp}

Assuming Hypothesis \ref{cyclic}, Proposition 11.4 of \cite{GHS} asserts that Hypothesis \ref{incorr} is valid when $p$ does not divide the order of the Weyl group of $G$.
In any case it has the following consequence, as in Corollary \ref{prin}:

\begin{cor}\label{unramprinc}  Suppose $\CLs(\pi)$ is pure (as Weil group parameter) and unramified.  Then $\pi$ is a constituent of an
unramified principal series representation.
\end{cor}

\begin{hyp}\label{Levi}  Let $M$ be a proper Levi subgroup of $G$.  The local Langlands correspondence for tempered representations, in Vogan's version Conjecture \ref{Vogan}, is known for $M$.  Moreover, the parametrization of tempered representations of $M$ is compatible with cyclic base change, in the sense of Property \ref{p10}.
\end{hyp}

\begin{hyp}\label{endo1}  Let $H$ be a proper elliptic endoscopic group of $G$.  The local Langlands correspondence for tempered representations, in Vogan's version Conjecture \ref{Vogan}, is known for $H$ and is compatible with cyclic base change and endoscopic transfer to $G$, in the sense of Properties  \ref{p9} and \ref{p10}.
\end{hyp}

\begin{hyp}\label{endo2}  Stable endoscopic transfer is known globally and locally for all endoscopic groups for $G$.
\end{hyp}

The local endoscopic transfer is assumed to satisfy the character identities of Property \ref{p9}.  By global endoscopic transfer we
mean the following.  Suppose $H$ is an endoscopic group of $G$ corresponding to the homomorphism of $L$-groups
$$i_H:   {}^LH \ra {}^LG.$$
Let $\Pi$ be a cuspidal automorphic representation of $H(\ad_K)$
for $K$ as above.  
Then there is a collection of automorphic representations $i_{H,*}(\Pi)$ of $G(\ad_K)$ such that,
for any $\Psi \in i_{H,*}(\Pi)$, 
$$\CLs(\Psi)_\ell \isoarrow  i_H \circ \CLs(\Pi)_\ell,$$
where $\CLs$ is as in \eqref{CLs} and $\isoarrow$ means up to conjugation by $\hat{G}(\Qlb)$.  It is also assumed that the local components of $\CLs(\Psi)$ at a prime $v$ of $K$
belong to the $L$-packet for $G(K_v)$ obtained by local endoscopic transfer from $H$.  


\section{Steps}

Let $\varphi:  W_F \ra \hG$ be an irreducible tempered parameter.   Our aim in this sketch is to show that the fiber $\CLs^{-1}(\varphi)$ satisfies the local Langlands conjecture
in Vogan's version.  It would be enough for the purposes of this sketch to show point (d) of Conjecture \ref{Vogan}, but the strategy is based on induction,
and the sketch will make it apparent that the parametrizations of Conjecture \ref{Vogan} (b) are necessary in order to carry out the intermediate inductive steps.

By Hypotheses \ref{Levi}, \ref{endo1}, and \ref{endo2}, we can assume that the component group of the centralizer of $\varphi$ is trivial; otherwise, the parameter $\varphi$ factors through a parameter for an endoscopic group, and then the structure
of $\CLs^{-1}(\varphi)$ is determined by the Hypotheses. 

\subsection{Step 1:  Globalization}

Let $\ell >> 0$ be a large prime number, $\ell \neq p$.  We follow the arguments in \cite[\S 9]{BHKT}, to which we refer for the following constructions.
We find a global  function field $K = k(X)$ with a place $v$ such that $K_v \isoarrow F$, and a surjective homomorphism
$$\br:  Gal(K^{sep}/K) \twoheadrightarrow \hG(\Fl)$$
such that $\br~|_{I_v} \isoarrow \phi~|_{I_v}$ where $I_v$ is the inertia group; this makes sense because $\ell$ is large.   As in \cite{BHKT}, we then lift $\br$ to a surjective parameter
$$\rho:  Gal(K^{sep}/K) \twoheadrightarrow \hG(\Zl)$$
with the desired local restriction to $I_v$.  This is  sufficient for our
purposes; the extension to the full decomposition group at $v$ is not difficult. 

\begin{hyp}\label{finmono}  We may do this in such a way that, for every place $w$ of $K$, the image of the inertia group $I_w$ under $\rho$
is finite.
\end{hyp}

This has not been checked but I believe it is not a problem.    It implies that the restriction of $\rho$ to the decomposition group at any point $w$
is a pure parameter.

We now apply the potential automorphy theorem of \cite{BHKT}.   After possibly replacing $K$ by a finite Galois extension $K'/K$, .  
we may assume that $\rho' := \rho~|_{Gal(K^{\prime,sep}/K')}$ is  again surjective and is the global parameter
attached by Lafforgue's construction to at least one {\it cuspidal automorphic} representation $\Pi'$.
Since this operation loses control of the ramification at the place $v$, by replacing $K'$ if necessary by an appropriate solvable extension
we may assume that $\rho'$ is everywhere unramified.   Since $\varphi$ was assumed irreducible, it is easy to see that we may also assume that $\rho'$ is pure of
weight $0$, so that $\Pi'$ is everywhere tempered.




\subsection{Step 2:  Unramified local packets}

The surjectivity of $\rho'$ implies that the image of $\rho'$ is Zariski dense, and we have assumed $\rho'$ is everywhere unramified. 
Let $\Pi_{\rho'}$ denote the set of automorphic representations  $\Pi'$ of $G(\ad_{K'})$ such that
$\CL(\Pi')_\ell = \rho'$; i.e., such that
$\rho' = \rho_{\Pi',\ell}$.  Corollary \ref{unramprinc}, together with Hypothesis \ref{finmono}, then implies that any $\Pi'\in \Pi_{\rho'}$ is globally unramified.  
By Hypothesis \ref{multone} this then implies that $\Pi' \in \Pi_{\rho'}$ has multiplicity one.  The set $\Pi'_{\rho'}$ is thus
determined by local data.  If we assume $G$ to be adjoint, then the classification of tempered unramified representations implies that $\Pi'_\rho$ is unique.

\begin{remark}  This is the only step that uses Hypothesis \ref{incorr}, and it is being used in the form of Corollary \ref{unramprinc}.  However, the corollary has
been proved in \cite{GHS} for $p$ that is sufficiently large  relative to the group, without reference to base change.  Thus in most cases Hypothesis \ref{incorr} 
is unnecessary. 
\end{remark}

Because $\rho'$ has Zariski dense image, we know that $\Pi_{\rho'}$ is orthogonal to the images of 
proper endoscopic transfer.  Thus it is globally stable.  Moreover, it can be separated from the non-cuspidal
contributions to the spectral side of the trace formula.

The unramified local packets are classified by representations of the $R$-group.
The following hypothesis should have a name:
\begin{hyp}\label{componentR}  Identify the $R$-group with the component group of the local parameter.
\end{hyp}

In the literature, this is  stated more generally:  Assume $M$ is a (rational) Levi subgroup of $G$ and $\sigma$ is a tempered
representation of $M(F)$.  Then there is an isomorphism between the Knapp-Stein $R$-group attached to
the pair $(M,\sigma)$ and the Arthur $R$-group attached to the pair $(\CL(\sigma),\hat{M})$.   We only have access
to $\CL^{ss}(\sigma)$ but in this induction we will only be considering Weil-Deligne parameters that are trivial on the $SL(2,\CC)$-factor.

\subsection{Strategy}

Let $\G_v \subset Gal(K'/K)$ be a decomposition group
for a prime of $K'$ over $v$, and let $L \subset K'$ be the fixed field of $\G_v$.  We use the same notation $v$ to denote the prime of $L$ fixed
by $\G_v$, and we know $L_v \isoarrow F$.   Since $\rho'$ extends to $Gal(K^{sep}/K)$, it extends
a fortiori to a (surjective) homomorphism $\rho_i:  Gal(K^{sep}/L_i) \ra \hG(\Zl)$ for every intermediate field $L \subset L_i \subset K'$.
Assume 
\begin{hyp}\label{oneplace} We can choose $K'$ in such a way as to guarantee that each such $\rho_i$ is unramified except at $v$.  
\end{hyp}

It's not clear it is always possible to satisfy Hypotheses \ref{finmono} and \ref{oneplace} simultaneously.  If one has to choose, the former
takes precedence, because a modification of the following argument, in which several places $v_i$ share the same decomposition group,
should lead to the same result.  However it is simpler to admit Hypothesis \ref{oneplace} in what follows.
In each case $Gal(K'/L_i)$ is then solvable; we limit our attention to the $L_i$ that are Galois over $L$.  We index the $L_i$ so that $i$ is the
number of prime factors in $|Gal(K'/L_i)|$ and write $L = L_N$.   We prove by induction that each $\rho_{L_i}$ is automorphic, and to identify
the set of automorphic representations $\{\Pi_i\}$ of $G(\ad_{L_i})$ with $\CL(\Pi_i) \isoarrow \rho_i$.    

We specifically 
identify the local components at the (unique) prime $v_i$ of $L_i$ dividing $v$; this is $\CL^{ss,-1}(\phi_i) = \Pi_{\phi_i}$, in the obvious
notation.  Under our inductive hypothesis ($\phi$ not endoscopic, etc.) we want $\Pi_{\phi} = \Pi_{\phi_N}$ to be a singleton.

\subsection{Step 3:  First Galois descent}

The first descent belongs to the principal series; in other words, it has abelian cuspidal
support.  Here is a crucial point:  we need to show that the only representations with  parameter contained in a (maximal) torus are
principal series, and more generally that the only representations with parameter contained in 
the Levi $\hat{M}$ of $\hG$ are constituents of induced representations from $M(F)$.  This can only be by
means of character identities.  We know that the constituents of representations induced from $M(F)$ already
suffice to satisfy the character identities.   Thus the stable characters of any other representations with the same
local parameter must contribute zero.   But the existence of any such additional representation should then
be ruled out by linear independence of characters.

\subsection{Subsequent steps}

Subsequent steps should follow the pattern of Arthur's induction in \cite{A}, and in particular inevitably involve the intricacies
of the {\it local intertwining relation} that is central to Arthur's argument.   In preparation for an analysis of the general case,
Beuzart-Plessis and I are looking at the simplest possible situation, when $N = 2$.


\begin{thebibliography}{99}

\bibitem[A]{A}  J. Arthur, {\it The endoscopic classification of representations--orthogonal and symplectic groups}, {\it American Mathematical Society Colloquium Publications}, {\bf 61} (2013).

\bibitem[AC]{AC}  J. Arthur, L. Clozel, {\it Simple algebras, base change, and the advanced theory of the trace formula}, {\it Annals of Math. Studies}, {\bf 120}, Princeton:  Princeton University Press (1989).  


\bibitem[BMY]{BMY}  A. Bertoloni Meli and A Youcis, An Approach to the characterization of the local Langlands correspondence, arXiv: 2003.11484 [math.NT].

\bibitem[Br08]{Br}  E. Brenner, Stability of the local gamma factor in the unitary case, {\it J. Number Theory}, {\bf 128} (2008) 1358--1375. 

\bibitem[BFHKT]{BFHKT}  G. B\"ockle, T. Feng, M. Harris, C. Khare,  J. Thorne, Cyclic base change of cuspidal automorphic representations over function fields, manuscript (2022).


\bibitem[BHKT]{BHKT}  G. B\"ockle, M. Harris, C. Khare,  J. Thorne, $\hat{G}$-local systems on smooth projective curves are potentially automorphic, {\it Acta Math.},  {\bf 223} (2019) 1--111.

\bibitem[B79]{Borel}  A. Borel, Automorphic L-functions. In {\it Automorphic forms, representations and L-functions} (Proc. Sympos. Pure Math., Oregon State Univ., Corvallis, Ore., 1977), Part 2, {\it Proc. Sympos. Pure Math.}, {\bf XXXIII}  (1979) 27--61. 

\bibitem[BH06]{BH06}  C. Bushnell, G. Henniart, {\it The local Langlands conjecture for $GL(2)$}. {\it Grundlehren der Mathematischen Wissenschaften}, {\bf 335} Springer-Verlag, Berlin, (2006). 

\bibitem[BH10]{BH10}  C. Bushnell, G. Henniart, The essentially tame local Langlands correspondence, III: the general case. {\it Proc. Lond. Math. Soc.},
{\bf 101} (2010) 497--553.

\bibitem[BHK]{BHK}   C. J. Bushnell, G. Henniart, P. Kutzko,  Correspondance de Langlands locale pour $GL_n$ et conducteurs de paires,
{\it Ann. Sci. \'Ecole Norm. Sup. }, {\bf 31} (1998) 537--560. 

\bibitem[BK]{BK}  C. Bushnell, P. Kutzko, {\it The admissible dual of $GL(N)$ via compact open subgroups}, {\it Annals of Math. Studies}, {\bf 129} Princeton University Press, Princeton, NJ, (1993). 

\bibitem[BK99]{BK99}   C. Bushnell, P. Kutzko, Semisimple Types in $GL_n$, {\it Compositio Mathematica}, {\bf 119} (1999) 53--97.

\bibitem[BG]{BG}  K. Buzzard, T. Gee,  The conjectural connections between automorphic representations and Galois representations, in {\it Automorphic forms and Galois representations.} Vol. 1, {\it London Math. Soc. Lecture Note Ser.}, {\bf 414}, Cambridge Univ. Press, Cambridge, (2014) 135--187. 

\bibitem[CFGK]{CFGK} Y. Cai, S. Friedberg, D. Ginzburg, E. Kaplan, Doubling constructions and tensor product $L$-functions: the linear case. {\it Invent. Math.},
{\bf 217} (2019)  985--1068.


\bibitem[Ch]{Ch}  G. Chenevier, Subgroups of $Spin(7)$ or $SO(7)$ with each element conjugate to some element of $G_2$ and applications to automorphic forms. {\it Doc. Math. }. {\bf 24} (2019) 95--161. 

\bibitem[CST]{CST}  J. Cogdell,  F. Shahidi,  T.-L. Tsai, Local Langlands correspondence for $GL_n$, and the exterior and symmetric square
$\varepsilon$-factors. {\it Duke Math. J.}, {\bf 166} (2017) 2053--2132.



\bibitem[CX]{CX}  C. Xue, Finiteness of cohomology groups of stacks of shtukas as modules over Hecke algebras, and applications. {\it \'Epijournal de G\'eom\'etrie Alg\'ebrique}, {\bf 4} (2020). 

\bibitem[DHKM]{DHKM}  J.-F. Dat, D. Helm, R. Kurinczuk, G. Moss, Moduli of Langlands Parameters,  	arXiv:2009.06708 [math.NT].

\bibitem[DR09]{DR09}  S.  DeBacker, M. Reeder, Depth-zero supercuspidal $L$-packets and their stability, {\it Annals of Math.}, {\bf 169} (2009) 795--901.


\bibitem[De73]{De73}  P. Deligne, Les constantes des \'equations fonctionnelles des fonctions $L$, in {\it Modular functions of one variable, II (Proc. Internat. Summer School, Univ. Antwerp, Antwerp, 1972)},{\it  Lecture Notes in Math.}, {\bf 349}, Springer, Berlin, (1973) 501--597. 

\bibitem[De80]{De80}  P. Deligne, La conjecture de Weil. II. {\it Inst. Hautes \'Etudes Sci. Publ. Math.}, {\bf 52} (1980) 137--252.

\bibitem[DKV]{DKV} P. Deligne, D. Kazhdan, M.-F. Vign\'eras:
Repr\'{e}sentations des alg\`{e}bres centrales simples
$p$-adiques, in J.-N.Bernstein, P.Deligne, D.Kazhdan, M.-F.Vigneras
{\it Repr\'{e}sentations des groupes r\'{e}ductifs sur un corps local},
Paris:  Hermann (1984).


\bibitem[FS]{FS}  L. Fargues, P. Scholze, Geometrization of the local Langlands correspondence, arXiv:2102.13459 [math.RT].

\bibitem[Fi19]{Fi19}  J. Fintzen, Tame cuspidal representations in non-defining characteristics,  	arXiv:1905.06374 [math.RT].

\bibitem[Fi21]{Fi21} J. Fintzen,  Types for tame $p$-adic groups,  {\it Annals of Math.}
{\bf 193}  (2021) 303--346.

\bibitem[FKS]{FKS}  J. Fintzen, T. Kaletha, L. Spice, A twisted Yu construction, Harish-Chandra characters, and endoscopy, \url{https://arxiv.org/abs/2106.09120v2}.

\bibitem[GHS]{GHS}  W.-T. Gan, M. Harris, W. Sawin, Local parameters of supercuspidal representations,  	arXiv:2109.07737 [math.RT]

\bibitem[GLo]{GLo} W.-T. Gan, L. Lomel\'i,  Globalization of supercuspidal representations over function fields and applications, {\it J. Eur. Math. Soc.}, {\bf 20} (2018) 2813--2858.

\bibitem[G]{G}  R. Ganapathy, The local Langlands correspondence for $GSp_4$ over local function fields, {\it Am. J. Mathematics}, {\bf 137} (2015) 1441--1534.

\bibitem[GV]{GV}  R. Ganapathy, S. Varma, On the local Langlands correspondence for split classical groups over local function fields. {\it J. Math. Inst. Jussieu}, {\bf 16} (2017) 987--1074.

\bibitem[GLa]{GLa} A. Genestier, V. Lafforgue, Chtoucas restreints pour les groupes r\'eductifs et param\'etrisation de Langlands locale,  	arXiv:1709.00978 [math.AG].
\newblock 


\bibitem[GRS]{GRS}  D. Ginzburg, S. Rallis,  D. Soudry, {\it The descent map from automorphic representations of $GL(n)$ to classical groups}, Singapore:  World Scientific (2011).

\bibitem[GH]{GH}  D. Ginzburg, J. Hundley, A doubling integral for $G_2$, {\it Israel J. Math.}, {\bf 207} (2015) 835--879. 

\bibitem[GJ72]{GJ72}  R. Godement, H. Jacquet, Zeta functions of simple algebras, {\it Lecture Notes in Math.}, {\bf 260} (1972).

\bibitem[Gr]{Gr}  R. L. Griess, Basic conjugacy theorems for $G_2$.  {\it Invent. Math.}, {\bf 121} (1995) 257--277.

\bibitem[GS15]{GS1}  N. Gurevich, A. Segal, The Rankin-Selberg integral with a non-unique model for the standard $L$-function of $G_2$, {\it J. Inst. Math. Jussieu}, {\bf 14} (2015)  149--184.

\bibitem[GS16]{GS2}  N. Gurevich, A. Segal, Poles of the standard $L$-function of $G_2$ and the Rallis-Schiffmann lift, Preprint (2016), available at https://arxiv.org/abs/1606.09069.



\bibitem[H97]{H97}  M. Harris, Supercuspidal representations in the cohomology of Drinfel'd upper half spaces; elaboration of Carayol's program. {\it Invent. Math.}, {\bf 129} (1997) 75--119.

\bibitem[H98]{H98}  M. Harris, The local Langlands conjecture for $GL(n)$ over a $p$-adic field, $n < p$. {\it Invent. Math. }, {\bf 134} (1998)  177--210.

\bibitem[HKS]{HKS}  M. Harris, S. Kudla, W. J. Sweet, Theta dichotomy for unitary groups, {\it J. Amer. Math. Soc.}, {\bf 9} (1996) 941--1004.

\bibitem[HKT]{HKT}  M. Harris, C. Khare, J. Thorne, A local Langlands parameterization for generic supercuspidal representations of $p$-adic $G_2$, 
{\it Ann. Sci. ENS}, to appear.  

\bibitem[HT01]{HT01}  M. Harris, R. Taylor, {\it The geometry and cohomology of some simple Shimura varieties.} With an appendix by Vladimir G. Berkovich. {\it Annals of Mathematics Studies}, {\bf 151} Princeton University Press, Princeton, NJ (2001).

\bibitem[HNY]{HNY} J. Heinloth, B.-C. Ng\^{o}, Z. Yun. Kloosterman sheaves for reductive groups, 
\it{Ann. of Math. } \textbf{177} (2013) 241--310.

\bibitem[He88]{He88} G. Henniart, La conjecture de Langlands locale num\'erique pour GL(n). {\it Ann. Sci. \'Ecole Norm. Sup.} {\bf 21} (1988)  497--544.

\bibitem[He90]{He90} G. Henniart, Une cons\'equence de la th\'eorie du changement de base pour GL(n). {\it Analytic number theory} (Tokyo, 1988), {\it Lecture Notes in Math.}, {\bf 1434}, Springer, Berlin, (1990) 138--142. 

\bibitem[He93]{He93}  G. Henniart, Caract\'erisation de la correspondance de Langlands locale par les facteurs $\epsilon$ de paires,
{\it Invent. Math. }, {\bf 113} (1993) 339--350.

\bibitem[He00]{He00}  G. Henniart, Une preuve simple des conjectures de Langlands pour GL(n) sur un corps p-adique. {\it Invent. Math. } {\bf 139} (2000) 439--455. 

\bibitem[He06]{He06}  G. Henniart, 
On the local Langlands and Jacquet-Langlands correspondences. {\it International Congress of Mathematicians. Vol. II} Eur. Math. Soc., Z\"urich, (2006) 1171--1182.

\bibitem[He21]{He21}  G. Henniart, Correspondance de Langlands et facteurs $\varepsilon$ des carr\'es ext\'erieur et sym\'etrique, manuscript (2021).  

\bibitem[HeLe]{HeLe} G. Henniart, B. Lemaire, Changement de base et induction automorphe pour $GL_n$ en caract\'eristique non nulle,
{\it M\'em. Soc. Math. Fr.}, {\bf 124} (2011). 

\bibitem[HeLo13]{HeLo13} G. Henniart, L. Lomel\'i, \emph{Uniqueness of Rankin-Selberg factors}, {\it J. Number Theory} {\bf 133} (2013) 4024--4035.

\bibitem[HeLo21]{HeLo21}  G. Henniart, L. Lomel\'i,  Asai cube $L$-functions and the local Langlands correspondence. {\it J. Number Theory} {\bf 221} (2021) 247--269.

\bibitem[HII08]{HII08}  K. Hiraga, A. Ichino, T. Ikeda, Formal degrees and adjoint $\gamma$-factors.  {\it J. Amer. Math. Soc.}, {\bf 21} (2008)  283--304.

\bibitem[HL]{HL}  J. Hundley, B. Liu, On automorphic descent from $GL_7$ to $G_2$, {\it J. Eur. Math. Soc.}, in press.

\bibitem[JPSS]{JPSS}  H. Jacquet, I. I. Piatetski-Shapiro,  J. Shalika:  Rankin-Selberg
convolutions, {\it Am. J. Math.}, {\bf 105}, 367-483 (1983).

\bibitem[JS]{JS}  H. Jacquet, J. Shalika, A lemma on highly ramified $\epsilon$-factors, {\it Math. Ann.}, {\bf 271}
(1985) 319--332.

\bibitem[JSo]{JSo}  D. Jiang, D. Soudry, The local converse theorem for $SO(2n+1)$ and applications, {\it Ann. of Math.}, 
{\bf 157} (2003) 743--806.

\bibitem[JSo2]{JSo2}  D. Jiang, D. Soudry, Appendix: On the local descent from $GL(n)$ to classical groups, {\it Amer. J. Math.},
{\bf 134} (2012) 767--772. 


\bibitem[Ka16]{Ka16}  T. Kaletha,  Rigid inner forms of real and $p$-adic groups, {\it Ann. of Math.}, {\bf 184} (2016) 559--632. 

\bibitem[Ka19]{Ka19}  T. Kaletha, Regular supercuspidal representations,  {\it J.Amer. Math. Soc.}, {\bf 32} (2019), 1071--1170. 

\bibitem[Ka20]{Ka20}  T. Kaletha, Supercuspidal $L$-packets, arXiv:1912.03274 [math.RT].

\bibitem[Ka22]{Ka22}  T. Kaletha, Representations of reductive groups over local fields, to appear in Proceedings of the ICM, 2022.

\bibitem[KMSW]{KMSW}  T. Kaletha,  A. M\'inguez, S. W. Shin, P.-J. White, Endoscopic Classification of Representations: Inner Forms of Unitary Groups,  arXiv:1409.3731 [math.NT]

\bibitem[KV]{KV}  D. Kazhdan, Y. Varshavsky, Endoscopic decomposition of certain depth zero representations, {\it Studies in Lie theory}, {\it Progr. Math.}, {\bf 243}, Birkh\"auser Boston, Boston, MA, (2006) 223--301.

\bibitem[Kim]{Kim}  J.-L Kim, Supercuspidal representations: an exhaustion theorem. {\it J. Amer. Math. Soc.}, {\bf 20} (2007)  273--320.

\bibitem[Ko02]{Ko02}  T. Konno, Twisted endoscopy and the generic packet conjecture, {\it Israel J. Math. }, {\bf 129} (2002) 253--289.

\bibitem[KS]{KS}  R. Kottwitz, D. Shelstad, Foundations of twisted endoscopy, {\it Ast\'erisque}, {\bf 255} (1999).

\bibitem[Ku]{Ku}  P. Kutzko, The Langlands conjecture for $Gl_2$ of a local field. {\it Ann. of Math.}, {\bf 112} (1980) 381--412.


\bibitem[Lab99]{Lab} J.-P. Labesse, Cohomologie, stabilisation, et changement de base, {\it Ast\'erisque}, {\bf 257} (1999).

\bibitem[LL]{LL}  J.-P. Labesse, B. Lemaire,  La formule des traces tordue pour les corps de fonctions, \url{https://arxiv.org/pdf/2102.02517v1.pdf}.

\bibitem[Laf02]{Laf02} L. Lafforgue,
Chtoucas de {D}rinfeld et correspondance de Langlands, {\it Invent. Math.}, {\bf 147} (2002) 1--241.


\bibitem[Laf18]{Laf18}  V. Lafforgue, Chtoucas pour les groupes r{\'e}ductifs et param{\'e}trisation de
 Langlands globale, {\it JAMS}, {\bf  31} (2018) 719--891.
  
 \bibitem[L97]{L97}  R. P. Langlands, Jr., Representations of abelian algebraic groups.  Olga Taussky-Todd, in memoriam,
{\it Pacific J. Math. } {\bf Special issue} (1997) 231--250. 
  
  \bibitem[LR05]{LR}  E. Lapid, S. Rallis, On the local factors of representations of classical groups, in J. W. Cogdell et al., eds, {\it Automorphic Representations, L-functions and Applications: Progress and Prospects}, Berlin:  de Gruyter (2005) 309--360.
  
 
 \bibitem[LRS]{LRS93}  G. Laumon, M. Rapoport,  U. Stuhler, D-elliptic sheaves and the Langlands correspondence, {\it Invent. Math.},  {\bf 113} (1993) 217--338.
 
 \bibitem[LH21]{LH21}  D. Li-Huerta, The local Langlands correspondence for $GL_n$ over function fields, \url{https://arxiv.org/abs/2106.05381}.
 
 \bibitem[Lo16]{Lo16}   L. Lomel\'i, On automorphic L-functions in positive characteristic, {\it Annales de l'Institut Fourier}, {\bf 66} (2016) 1733--1771. 
 
\bibitem[Lo19]{Lo19}  L. Lomel\'i,  Rationality and holomorphy of Langlands-Shahidi $L$-functions over function fields, {\it Math. Z.} {\bf 291} (2019) 711--739.

\bibitem[MY]{MY}  A. Mayeux, Y. Yamamoto, Comparing Bushnell-Kutzko and S\'echerre's constructions of types for
 $\mathrm{GL}_{N}$ and its inner forms with Yu's construction, arXiv:2112.12367 [math.RT]
 
 \bibitem[Mok]{Mok}, C.-P. Mok, Endoscopic classification of representations of quasi-split unitary groups, {\it Mem. Amer. Math. Soc.}, {\bf 235} (2015), no. 1108.


\bibitem[PSR87]{psr}  S. Rallis, I. I. Piatetski-Shapiro, {\it $L$-functions for the classical groups}, in {\it Explicit constructions of automorphic $L$-functions}, {\it Lecture Notes in Math.}, {\bf 1254} Springer-Verlag, Berlin, (1987) 1--52. 

\bibitem[RS05]{RS}  S.  Rallis,  D.  Soudry,  Stability  of  the  local  gamma  factor  arising  from  the  doubling  method,  {\it Math.
Ann.}, {\bf 333} (2005) 291--313. 

\bibitem[OT]{OT}  M. Oi, K. Tokimoto, Local Langlands Correspondence for Regular Supercuspidal Representations of $GL(n)$, {\it Int. Math. Res. Notices}, 
{\bf 2021}, (February 2021) 2007--2073.

\bibitem[R]{R} J. Rogawski:  Representations of $GL(n)$ and division algebras
over a $p$-adic field, {\it Duke Math. J.}, {\bf 50} (1983) 161--169.

\bibitem[Sch13]{Sch13} P. Scholze, The local Langlands correspondence for $GL_n$ over $p$-adic fields. {\it Invent. Math.}, {\bf 192} (2013) 663--715.

\bibitem[SS]{SS}  V. S\'echerre, S. Stevens, Smooth representations of $GL_m(D)$ VI: semisimple types. {\it Int. Math. Res. Not.},
{\bf IMRN 2012}  (2012) 2994--3039. 

\bibitem[Se]{Se}  A. Segal, A family of new-way integrals for the standard $L$-function of cuspidal representations of the exceptional group of type $G_2$,
{\it Int. Math. Res. Not.}, {\bf IMRN 2017} 2014--2099. 

\bibitem[Sh83]{Sh83} F. Shahidi:  Local coefficients and normalization of intertwining
operators for $GL(n)$, {\it Comp. Math.}, {\bf 48} (1983) 271--295.

\bibitem[Tam16]{Tam16}  K.-F. Tam, Admissible embeddings of $L$-tori and the essentially tame local Langlands correspondence, {\it Int. Math. Res. Not.},
{\bf IMRN 2016}  1695--1775. 


\bibitem[T79]{T79}  J. Tate, Number theoretic background, in {\it  Automorphic forms, representations and L-functions,  Part 2}, {\it Proc. Sympos. Pure Math.}, {\bf XXXIII}, Amer. Math. Soc., Providence, R.I., (1979) ,  3--26.

\bibitem[Tu78]{Tu78}  J. Tunnell, On the local Langlands conjecture for $GL(2)$, {\it Invent. Math.}, {\bf 46} (1978) 179--200. 

\bibitem[V93]{V93}  D. Vogan, The local Langlands conjecture, {\it Contemporary Math.}, {\bf 145}  (1993)  305--379.

\bibitem[Ya14]{Ya14}  S. Yamana, $L$-functions and theta correspondence
for classical groups, {\it Invent math}, {\bf 196} (2014) 651--732.

\bibitem[Yu01]{Yu01}  J.-K. Yu, Construction  of  tame  supercuspidal  representations,   {\it J.Amer. Math. Soc.}, {\bf 14} (2001) 579--622.

\bibitem[Yu09]{Yu09}  J.-K. Yu, Bruhat-Tits theory and buildings, {\it Ottawa Lectures on Admissible Representations of Reductive
p-Adic Groups}, {\it Fields Inst. Monogr.}, {\bf 26}, Amer. Math. Soc., Providence, RI (2009) 53--77.

\bibitem[Z80]{Z80}  A. V. Zelevinsky, Induced representations of reductive $p$-adic groups II. On irreducible representations of $GL(n)$. 
{\it Ann. Sci. Ec. Norm. Sup\'er.}, {\bf 13} (1980) 165--210.

\bibitem[Zhu]{Zhu} X. Zhu,  Coherent sheaves on the stack of Langlands parameters,  https://arXiv:2008.02998v2 [math.AG].

\end{thebibliography}
\end{document}